\documentclass{amsart}
\usepackage{graphicx} % Required for inserting images

\title[Hessian stability and convergence rates for entropic and Sinkhorn potentials]{Hessian stability and convergence rates for entropic and Sinkhorn potentials via semiconcavity}

\author{Giacomo Greco}
\address{Università degli studi di Roma Tor Vergata}
\curraddr{RoMaDS - Department of Mathematics, 00133RM Rome, Italy.}
\email{greco@mat.uniroma2.it}
\thanks{We warmly thank Alberto Chiarini and Giovanni Conforti for suggesting this problem and for many useful discussions. Both authors are associated to INdAM (Istituto Nazionale di Alta
Matematica ``Francesco Severi'') and the group GNAMPA.
GG is supported by the PRIN project GRAFIA (CUP:
E53D23005530006) and MUR Departement of Excellence Programme 2023-2027 MatMod@Tov (CUP: E83C23000330006). GG is thankful to Università Cattolica del Sacro Cuore for its hospitality and where part of this work has been carried out. LT acknowledges financial support by the INdAM-GNAMPA Project 2024 ``Mancanza di regolarità e spazi non lisci: studio di autofunzioni e autovalori'' (CUP: E53C23001670001).}

\author{Luca Tamanini}
\address{Università Cattolica del Sacro Cuore}
\curraddr{Dipartimento di Matematica e Fisica ``Niccol\`o Tartaglia'', I-25133 Brescia, Italy.}
\email{luca.tamanini@unicatt.it}
\RequirePackage{amsthm,amsmath,amsfonts,amssymb}

\usepackage{makecell,multirow}
\usepackage{stmaryrd}
\DeclareMathAlphabet{\mathpzc}{OT1}{pzc}{m}{it}
\usepackage{url}
\usepackage{graphicx}
\usepackage{mathtools}
\usepackage[dvipsnames]{xcolor}\usepackage{bbm,bm}
\usepackage{mathrsfs}
\usepackage{makecell}
\usepackage{diagbox}
\usepackage{amssymb}
\usepackage{amsfonts}
\usepackage{upgreek}
\usepackage{nicefrac}
\usepackage{ifthen}
\usepackage{xargs}
\usepackage{aliascnt}
\usepackage{cleveref}
\usepackage{enumitem}

\numberwithin{equation}{section}
\theoremstyle{plain}
\newtheorem{theorem}{Theorem}[section]
\crefname{theorem}{theorem}{Theorems}
\Crefname{Theorem}{Theorem}{Theorems}

\newaliascnt{lemma}{theorem}
\newtheorem{lemma}[lemma]{Lemma}
\aliascntresetthe{lemma}
\crefname{lemma}{lemma}{lemmas}
\Crefname{Lemma}{Lemma}{Lemmas}

\newaliascnt{corollary}{theorem}
\newtheorem{corollary}[corollary]{Corollary}
\aliascntresetthe{corollary}
\crefname{corollary}{corollary}{corollaries}
\Crefname{Corollary}{Corollary}{Corollaries}

\newaliascnt{proposition}{theorem}
\newtheorem{proposition}[proposition]{Proposition}
\aliascntresetthe{proposition}
\crefname{proposition}{proposition}{propositions}
\Crefname{Proposition}{Proposition}{Propositions}
\theoremstyle{definition}
\newtheorem{definition}{Definition}
\crefname{definition}{definition}{definitions}
\Crefname{Definition}{Definition}{Definitions}

\newtheorem{remark}[definition]{Remark}
\crefname{remark}{remark}{remarks}
\Crefname{Remark}{Remark}{Remarks}

\crefname{example}{example}{examples}
\Crefname{Example}{Example}{Examples}

\newtheorem{assumption}{\textbf{H}\hspace{-3pt}}
\Crefname{assumption}{\textbf{H}\hspace{-3pt}}{\textbf{H}\hspace{-3pt}}
\crefname{assumption}{\textbf{H}}{\textbf{H}}

\Crefname{assumptionA}{\textbf{A}\hspace{-3pt}}{\textbf{A}\hspace{-3pt}}
\crefname{assumptionA}{\textbf{A}}{\textbf{A}}
\newcommand{\bes}{\begin{equation*}}
\newcommand{\ees}{\end{equation*}}
\newcommand{\beas}{\begin{eqnarray}}
\newcommand{\eeas}{\end{eqnarray}}
\newcommand{\bea}{\begin{eqnarray}}
\newcommand{\eea}{\end{eqnarray}}
\newcommand{\be}{\begin{equation}}
\newcommand{\ee}{\end{equation}}
\newcommand{\bei}{\begin{itemize}}
\newcommand{\eei}{\end{itemize}}
\newcommand{\bec}{\begin{cases}}
\newcommand{\eec}{\end{cases}}
\newcommand{\ben}{\begin{enumerate}}
\newcommand{\een}{\end{enumerate}}

\DeclareMathOperator{\Tr}{Tr}

\DeclareMathOperator{\supp}{supp}

\def\PE{\mathbb{E}}

\def\txts{\textstyle}

\def\bfW{\mathbf{W}}

\def\rmL{\mathrm{L}}

\newcommand{\IND}{\mathbf{1}}
\newcommand{\bbR}{\mathbb{R}}

\newcommand{\bbRD}{\mathbb{R}^d}

\newcommand{\N}{\mathbb{N}}

\newcommand{\cP}{\mathcal{P}}

\newcommand{\cT}{\mathcal{T}}

\newcommand{\cL}{\mathcal{L}}
\newcommand{\scrH}{\mathscr{H}}

\newcommand{\bbP}{\mathbb{P}}
\newcommand{\bbE}{\mathbb{E}}

\newcommand{\bbl}{\begin{block}}
\newcommand{\ebl}{\end{block}}

\newcommand{\De}{\mathrm{d}}

\newcommandx{\Vnorm}[2][1=V]{\| #2 \|_{#1}}
\newcommandx{\norm}[2][1=]{\ifthenelse{\equal{#1}{}}{\left\Vert #2 \right\Vert}{\left\Vert #2 \right\Vert^{#1}}}
\newcommandx{\normLigne}[2][1=]{\ifthenelse{\equal{#1}{}}{\Vert #2 \Vert}{\Vert #2\Vert^{#1}}}

\def\rset{\mathbb{R}}

\def\nset{\mathbb{N}}

\def\ie{\textit{i.e.}}

\def\eqsp{\;}

\def\rmd{\mathrm{d}}

\def\Leb{\mathrm{Leb}}

\def\Leb{\mathrm{Leb}}

\newcommand{\Lam}[1]{\Lambda(#1)}
\newcommand{\Cphi}{C^{\varphi^\nu}} %lasciateli per il momento così sennò non compila più per doppi subscript e superscripts  (e soprattutto ci metto un sacco a trovare l'errore di compilazione)
\newcommand{\Cpsi}{C^{\psi^\nu}}
\newcommand{\Crhonu}{C_{\rho\nu}}
\newcommand{\Crhomu}{C_{\rho\mu}}

\newcommand{\HS}[1]{\|#1\|_{\mathrm{HS}}}
\newcommand{\Kappadel}{K_{\delta'\delta}^{\rho\nu}}
\newcommand{\Kappadelmu}{K_{\delta'\delta}^{\rho\mu}}

\newcommand{\SymIpsi}[3]{\mathcal{I}^{#3}(#1,#2)}

\begin{document}

\begin{abstract}
In this paper we determine quantitative stability bounds for the Hessian of entropic potentials, \ie, the dual solution to the entropic optimal transport problem. To the authors' knowledge this is the first work addressing this second-order quantitative stability estimate in general unbounded settings. Our proof strategy relies on semiconcavity properties of entropic potentials and on the representation of entropic transport plans as laws of forward and backward diffusion processes, known as Schr\"odinger bridges. Moreover, our approach allows to deduce a stochastic proof of quantitative stability estimates for entropic transport plans and for gradients of entropic potentials as well. Finally, as a direct consequence of these stability bounds, we deduce exponential convergence rates for gradient and Hessian of Sinkhorn iterates along Sinkhorn's algorithm, a problem that was still open in unbounded settings. Our rates have a polynomial dependence on the regularization parameter.

\medskip

\end{abstract}

\maketitle

%%%% Main text entry area:

%%%%% intro file %%%%%%%%

\section{Introduction}
Given two probability measures $\rho,\,\mu \in \cP(\mathbb{R}^d)$ and a regularization parameter $T>0$, the Entropic Optimal Transport problem (EOT henceforth) reads as
\bes
\text{minimize } \int_{\mathbb{R}^d \times \mathbb{R}^d}
\frac{|x-y|^2}{2}\,\De\pi+T\,\scrH(\pi|\rho\otimes\mu) \text{
  under the constraint } \pi\in\Pi(\rho,\mu)\eqsp,
\ees
where $\scrH$ denotes the relative entropy functional (aka Kullback--Leibler divergence) and $\Pi(\rho,\mu)$ is the set of couplings of $\rho$ and $\mu$. This problem can be seen as an entropic regularization of the Optimal Transport (OT) problem, which indeed is recovered in the limit case $T=0$. For this reason, EOT has been widely studied in the last years and the solutions to its primal and dual formulation are respectively used as proxies for optimal transport plans and Brenier's optimal transport map \cite{Mikami04, BerntonGhosalNutz, nutz2021entropic, lagg2022gradient}. Lastly, EOT is equivalent to a statistical mechanics problem, known as the Schr\"odinger problem, introduced in \cite{Schr,Schr32} where E.\ Schr\"odinger was interested in the most likely evolution of a cloud of Brownian particles, conditionally to its initial and final distribution at time $s=0$ and $s=T$ respectively. Therefore EOT has a cutting-edge nature that lies at the interface between analysis and stochastics. Moreover, this problem has recently gained more popularity due to its use in machine learning and generative modeling applications \cite{bortoli2021diffusion,wang2021DeepGenviaSB,shi2022ConditionalSimviaSB}, mainly due to the possibility of solving EOT via an iterative algorithm, known as \emph{Sinkhorn's algorithm} \cite{Sinkhorn64,SinkhornKnopp67} or \emph{Iterative Proportional Fitting Procedure}, which can be used to quickly obtain approximate solutions for EOT \cite{cuturi2013sinkhorn} in a much easier and faster way, compared to standard OT solvers.

\medskip

In this article, we are interested in analyzing how changes in the marginals $\rho,\mu$ affect solutions to EOT. By relying on semiconcavity bounds and stochastic calculus, we are going to show below quantitative stability estimates for EOT potentials up to the second order, namely for their gradient and Hessian. To the best of our knowledge, this is the first work where second-order quantitative stability estimates are obtained. This is even more remarkable when compared with unregularized optimal transport, where higher-order quantitative stability estimates are more difficult to obtain, as in general potentials may lack regularity and the Ma--Trudinger--Wang condition \cite{MTW05} is imposed in order to ensure it; without this demanding assumption, only first-order quantitative stability bounds are available in general (see for instance the very recent \cite{letrouit2024gluingmethodsquantitativestability,kitagawa2025stabilityoptimaltransportmaps} and references therein). On the contrary, our main stability theorem is valid under fairly general assumptions, significantly weaker than the Ma--Trudinger--Wang condition, and since EOT is used as a proxy for OT, this highlights the importance of our result.

\medskip

In order to continue the exposition and state clearly our main contributions, let us collect a few basic facts about EOT and its solutions. First, let us recall that under mild assumptions on the marginals $\rho,\mu$ (see for instance \cite[Proposition 2.2]{lagg2022gradient}), EOT admits a unique minimizer $\pi^\mu\in\Pi(\rho,\mu)$, referred to as the \emph{entropic plan} (or Schr\"odinger plan), and there exist two functions $\varphi^\mu \in \mathrm{L}^1(\rho)$ and $\psi^\mu \in \mathrm{L}^1(\mu)$, called \emph{entropic potentials}, such that
\bes
\pi^\mu(\De x\De y)=(2\pi
T)^{-d/2}\,\exp\biggl(-\frac{|x-y|^2}{2T}-\varphi^\mu(x)-\psi^\mu(y)\biggr)\,\De
x\,\De y\eqsp.
\ees
Both the optimal plan $\pi^\mu$ and the entropic potentials $\varphi^\mu,\psi^\mu$ depend on $T$ and on $\rho$, but for ease of notation we omit this dependence, as $T$ and $\rho$ will be kept fixed throughout the whole manuscript, whereas we are interested in stability bounds for changes in the second marginal in EOT. The pair $(\varphi^\mu,\psi^\mu)$ is unique up to constant translations $a\mapsto (\varphi^\mu+a,\psi^\mu -a)$ and it is characterized as solution to a system of equations. Indeed, if we suppose that the marginals admit densities of the form
\bes
\rho(\De
x)=\exp(-U_\rho(x))\De x\,,\qquad\mu(\De y)=\exp(-U_\mu(y))\De y\eqsp,
\ees
then, imposing that $\pi^\mu\in\Pi(\rho,\mu)$ one finds that $\varphi^\mu,\psi^\mu$ solve the following system of implicit functional equations, known as \emph{Schr\"odinger system}
\be\label{schrodinger:system}
   \varphi^\mu= U_\rho + \log P_T \exp(-\psi^\mu)\,,\qquad
   \psi^\mu= U_\mu + \log P_T \exp(-\varphi^\mu)\eqsp,
\ee
where $(P_s)_{s\geq 0}$ is the Markov semigroup generated by the standard $d$-dimensional Brownian motion $(B_s)_{s \geq 0}$, defined as $P_sf(x) = \PE[f(x+B_s)]$ for any non-negative measurable function $f:\rset^d\to \rset$. 

The structure of the Schr\"odinger system motivates the introduction of the ``interpolated potentials''
\bes
\varphi^\mu_s = -\log P_{T-s}\exp(-\varphi^\mu)\,, \qquad \psi^\mu_s = -\log P_{T-s}\exp(-\psi^\mu) \,.
\ees
It is easily seen that they are solutions to the backward Hamilton--Jacobi--Bellman equation  
\be\tag{HJB}\label{eq:HJB}
\partial_s u_s+\frac12\Delta u_s-\frac12|\nabla u_s|^2=0
\ee
with final conditions $u_T = \varphi^\mu$ and $u_T = \psi^\mu$ respectively. Such a PDE enjoys a fundamental property of backpropagation of convexity (see \Cref{lemma:backpropagation:convexity} in the Appendix) and this has recently been employed in a stochastic analysis framework in order to prove convexity/concavity estimates for entropic potentials in \cite{conforti2024weak}, providing an entropic version of the celebrated Caffarelli Theorem for Lipschitzianity of transport maps (see also \cite{chewi2022entropic,fathi2019proof} for a non-stochastic proof).
As shown in \cite{conforti2023Sinkhorn,chiarini2024semiconcavity}, semiconcavity estimates play a pivotal role in establishing entropic quantitative stability results. In this work, we continue the research line started there, where semiconcavity was used for entropic stability of entropic plans and exponential convergence of Sinkhorn's algorithm; here, we focus on quantitative stability bounds for gradient and Hessian of entropic potentials. For these reasons, let us  introduce the notion of semiconcavity that we employ in our paper. We say that a function $f:\bbRD\rightarrow \mathbb{R}$ is $\Lambda$-semiconcave if for all $z,y\in \bbRD$ we have
\be\label{def:Lsemiconcave}
f(z)-f(y)\leq \langle\nabla f(y),\,z-y\rangle+\frac{\Lambda}{2}\,|z-y|^2\,.
\ee
As already observed in \cite{conforti2023Sinkhorn,chiarini2024semiconcavity}, a crucial role is played by the semiconcavity of the function
\be\label{eq:def:costo+funzione}
g^y_h(z)\coloneqq \frac{|z-y|^2}{2}-T\,h(z)
\ee
where $h\in\{\varphi^\mu_0,\,\psi^\mu_0\}$ is a backpropagated entropic potential along~\ref{eq:HJB}.
We will denote with $\Lam{h}$ a semiconcavity parameter of $g^y_h$ (uniform in $y$). To be more precise, in our examples and in the explicit computations we will fix a parameter $\Lambda\in\rset$ such that~\eqref{def:Lsemiconcave} holds. We do not assume it to be the optimal parameter choice.

\medskip

We are now ready to state our main assumptions and results:

\begin{assumption}\label{ass:basic}
    Let us assume that $\rho,\,\mu\in \cP_2(\mathbb{R}^d)$ have finite relative entropy, namely $\scrH(\rho|\Leb) < +\infty$ and $\scrH(\mu|\Leb) < +\infty$.
\end{assumption}

This first assumption is standard in EOT when considering its Schr\"odinger problem formulation and it guarantees the existence and uniqueness of optimal plan, entropic potentials as well as the validity of the stochastic representation via forward-backward Schr\"odinger bridge processes, as described in \Cref{sec:SB:processes} below. The second assumption is needed when introducing a different marginal $\nu\in\cP(\bbRD)$.

\begin{assumption}\label{ass:mu}
    Assume that $\nu\in \cP_2(\mathbb{R}^d)$ has finite relative entropy, namely $\scrH(\nu|\Leb) < +\infty$. Moreover, let us assume that: (a) either $\scrH(\mu|\nu)<+\infty$; (b) or $\mu\ll\nu$ and $\Lam{\varphi^\mu_0}$ is finite.
\end{assumption}

\begin{remark}
Let us stress that, despite that the finiteness of $\Lam{\varphi^\mu_0}$ in \Cref{ass:mu} may seem as a condition on $\mu$, there exist sufficient conditions on $\rho$ that ensure its validity without any extra assumption on $\mu$. For instance, the compactness of the support of $\rho$ or the log-concavity of its Radon--Nikod\'ym derivative, as shown by the computations performed in the Appendix.
\end{remark}

Under these assumptions, we will prove a general Hessian (and gradient) stability result which builds upon semiconcavity estimates for $\Lam{\varphi^\nu_0}$. In order to show its wide validity, we will further specialize these general estimates in two landmark examples: compactly supported and log-concave marginals. By building upon estimates obtained in \cite{chiarini2024semiconcavity} our quantitative stability estimates could be applied to weakly log-concave marginals or could be further specialized to the more regular Caffarelli's setting (namely when the Hessian of marginals' log-densities are both upper and lower bounded). For sake of exposition, we have omitted these two applications where the constants are less readable. In what follows, whenever we write that a constant depends polynomially on a measure $\rho\in\cP(\bbRD)$, we mean that it depends at most polynomially on the geometric parameters of $\rho$, such as the diameter of the support or the log-concavity parameter.
Our main stability result reads as follows.

\begin{theorem}[Informal main result]\label{thm:intro}
Assume \Cref{ass:basic} and \Cref{ass:mu}. We have
    \bes 
\|\nabla\varphi^\nu-\nabla\varphi^\mu \|^2_{\rmL^2(\rho)} \lesssim \bfW_2^2(\mu,\nu)\quad\text{ and }\quad
    \| \nabla^2 \varphi^\mu -  \nabla^2 \varphi^\nu\|_{\rmL^1(\rho)} \lesssim 
  \bfW_2(\mu,\nu) + \bfW_2^2(\mu,\nu)\,,
    \ees
up to multiplicative constants that depend polynomially only on $\rho,\,\nu,\,T$ (and not on $\mu$), and are explicit. Moreover, if we specify our result to the following settings we deduce that
\begin{itemize}
        \item If \Cref{ass:basic} holds, $\scrH(\nu|\Leb)<\infty$, $\supp(\rho),\,\supp(\nu)\subseteq B_R(0)$ (for some radius big enough, \ie, $R^2\geq T$) and either $\mu\ll\nu$ or $\supp(\mu)\subseteq B_R(0)$, then 
        \bes 
        \begin{aligned}
            &\|\nabla\varphi^\nu-\nabla\varphi^\mu \|^2_{\rmL^2(\rho)} \lesssim\nicefrac{R^4}{T^4}\,\bfW_2^2(\mu,\nu)\,,\\
    &\| \nabla^2 \varphi^\mu -  \nabla^2 \varphi^\nu\|_{\rmL^1(\rho)} \lesssim
   (\nicefrac{R^4}{T^{\nicefrac72}}+\nicefrac{d}{T})\,\bfW_2(\mu,\nu)+
   \nicefrac{R^6}{T^5}\,\bfW_2^2(\mu,\nu)\,, 
   \end{aligned}
    \ees
    \item If \Cref{ass:basic} holds, $\scrH(\nu|\Leb)<\infty$, and both $\rho$ and $\nu$ are log-concave, \ie, their (negative) log-densities satisfy  $
 \nabla^2U_\rho\geq \alpha_\rho$  and $\nabla^2U_\nu\geq \alpha_\nu$ for some  $\alpha_\rho,\,\alpha_\nu>0$ (wlog such that $\alpha_\rho\vee\alpha_\nu<T^{-1}$), then 
   \bes 
        \begin{aligned}
            &\|\nabla\varphi^\nu-\nabla\varphi^\mu \|^2_{\rmL^2(\rho)} \lesssim \frac{1}{\alpha_\rho\,\alpha_\nu\,T^4}\,\bfW_2^2(\mu,\nu)\,,\\
    &\| \nabla^2 \varphi^\mu -  \nabla^2 \varphi^\nu\|_{\rmL^1(\rho)} \lesssim
   \biggl(\frac1{\alpha_\nu\,\sqrt{\alpha_\rho}\,T^3}+\frac{d}{\sqrt{\alpha_\rho\,\alpha_\nu}\,T^2}\biggr)\,\bfW_2(\mu,\nu)+
   \frac1{\alpha_\rho\,\alpha_\nu\,T^4}\,\bfW_2^2(\mu,\nu)\,.
   \end{aligned}
    \ees
    \end{itemize}
\end{theorem}

In this paper, whenever we write the $\rmL^1$-norm of a matrix we are considering the $\rmL^1$-norm of its Hilbert--Schmidt norm, the latter being defined as $
\HS{A}^2 = \sum_{i,j} A_{i,j}^2$. 
The presence of the dimension $d$ in these last second-order bounds comes from the Hilbert--Schmidt norm. Indeed, all our estimates are dimension-free up to being able to control the Hilbert--Schmidt norm of the Hessian of backpropagated potentials $(\nabla^2\psi^\nu_s)_{s\in[0,T)}$. In order to bound these last norms, in the Appendix, we rely on the known identity $\HS{\nabla^2\psi^\nu_s}\leq \sqrt{d}\norm{\nabla^2\psi^\nu_s}_2$ which allows us to efficiently bound this Hilbert--Schmidt norm in terms of the semiconcavity parameter $\Lam{\psi^\nu_0}$.

\begin{remark}
   We have stated our general Wasserstein first- and second-order stability estimates under the absolute continuity assumption $\mu\ll\nu$. This restriction is mainly due to the stochastic control and entropic strategy we have employed in the proof of our main result. In particular, this is a necessary condition for the pivotal bound~\eqref{eq:bound:conditionals:W2} we prove for conditional Schr\"odinger bridges. Despite that, our Wasserstein stability estimates can be extended via a regularization argument to a more general setting where $\mu\not\ll\nu$, provided one is able to control the semiconcavity parameter alongside the regularization procedure. This requires a case-by-case discussion. In the compactly supported case we show that if we further assume $\supp(\mu)\subseteq B_R(0)$, then in~\Cref{ass:mu} we may  drop $\mu\ll\nu$ (which is a stronger assumption). Moreover, we show via a heat kernel regularization that $\mu\ll\nu$ in~\Cref{ass:mu} can be completely dropped in the log-concave case. For the interested reader, these regularization procedures are performed in the proof of Corollaries B.1, B.2, B.3, and B.4 in the Appendix, where we further compute explicit rates and constants for these two specific settings.
\end{remark}

Let us also comment on a hidden technical point and a first reason why the previous statement is ``informal''. Rather than the differences $\nabla\varphi^\nu - \nabla\varphi^\mu$ and $\nabla^2\varphi^\nu - \nabla^2\varphi^\mu$, in Theorem \ref{thm:intro} we control $\nabla(\varphi^\nu - \varphi^\mu)$ and $\nabla^2(\varphi^\nu - \varphi^\mu)$. Note indeed that no regularity assumptions are formulated on $\rho$, so that $\varphi^\nu,\varphi^\mu$ may lack the required regularity. However, $\varphi^\nu - \varphi^\mu = \psi^\mu_0 - \psi^\nu_0$, which is instead the difference of two solutions to \ref{eq:HJB}, hence of two regular functions. Moreover, under some regularity assumption on $\rho$ (e.g.\ $\rho \in C^2(\mathbb{R}^d)$), gradient and Hessian of $\varphi^\mu,\varphi^\nu$ are in fact well defined. 

\medskip

For the reader's sake, we collect here the references within this article where our informal main result is stated and proven. The quantitative stability bound for gradients is proven in \Cref{thm:stability:mappe:general} whereas the Hessian stability bound is proven in \Cref{thm:stability_hessians:general}, where the explicit constants are expressed in terms of $T$ and of the semiconcavity and geometric parameters of $\rho,\nu$. The above statement is informal also for a second reason: solely under \Cref{ass:basic} and \Cref{ass:mu}, it is not clear whether these constants are finite, although we are able to show it and compute their asymptotics in our specialized setting. In particular, the compact setting and log-concave bounds are based on the explicit computations we perform in  \Cref{cor:grad:compact}, \Cref{cor:grad:log:conc}, \Cref{cor:hess:compact} and  \Cref{cor:hess:log:conc} in the Appendix.

\subsection{Exponential convergence of Hessian of Sinkhorn's iterates}
Most of the popularity EOT has recently gotten is due to the possibility of rapidly computing its solutions via an iterative algorithm, known as \emph{Sinkhorn's algorithm} \cite{Sinkhorn64,SinkhornKnopp67} or \emph{Iterative Proportional Fitting Procedure (IPFP)}. Given any initialization $\varphi^{0}\colon\bbRD\to\rset$, this algorithm solves \eqref{schrodinger:system} as a fixed point problem by generating two sequences $\{\varphi^n,\psi^n\}_{n \in\nset}$, called \emph{Sinkhorn potentials}, defined recursively as:
\bes
   \varphi^{n+1}= U_\rho + \log P_T \exp(-\psi^{n})\,,\qquad
   \psi^{n+1}= U_\mu + \log P_T \exp(-\varphi^{n+1})\eqsp.
\ees
As pointed out in~\cite{benamou2015iterative}, this is also equivalent to Bregman's iterated projection algorithm for relative entropy, which in the current setup produces two sequences of plans $(\pi^{n,n},\pi^{n+1,n})_{n\in\nset}$ starting from a positive measure $\pi^{0,0}$ according to the following recursion:
\bes\label{eq:primal_pb}
\pi^{n+1,n}\coloneqq{\arg\min}_{\Pi(\rho,\star)} \scrH(\cdot|\pi^{n,n})\,, \qquad\pi^{n+1,n+1}\coloneqq {\arg\min}_{\Pi(\star,\mu)} \scrH(\cdot|\pi^{n+1,n})\,,
\ees
where $\Pi(\rho,\star)$ (resp.\ $\Pi(\star,\mu)$) is the set of probability measures $\pi$ on $\mathbb{R}^{2d}$ such that the first marginal is $\rho$, \ie, $(\mathrm{proj}_x)_{\sharp}\pi = \rho$  (resp.\ the second marginal is $\mu$, \ie, $(\mathrm{proj}_y)_{\sharp}\pi = \mu$).
It is relatively easy (cf. \cite[Section 6]{Marcel:notes}) to show that, starting from $\pi^{0,0}(\rmd x \rmd y) \propto \exp(-\nicefrac{|x-y|^2}{2T}-\psi^{0}(y) - \varphi^{0}(x)) \rmd x \rmd y$, the  iterates in \eqref{eq:primal_pb} are related to Sinkhorn potentials through 
\begin{align*}
\txts \pi^{n+1,n}(\De x \De y) \txts \propto e^{-\frac{|x-y|^2}{2T}-\varphi^{n+1}(x)-\psi^n(y)} \rmd x \rmd y \eqsp, \qquad
\txts   \pi^{n+1,n+1}(\De x\De y)  \txts \propto e^{-\frac{|x-y|^2}{2T}-\varphi^{n+1}(x)-\psi^{n+1}(y)}\rmd x \rmd y \eqsp.
\end{align*}
In the sequel, we will refer to the couplings $(\pi^{n,n},\pi^{n+1,n})_{n\in\nset}$ as Sinkhorn plans. By definition $\pi^{n+1,n}$ has the correct first marginal, but wrong second marginal, which we denote with  $\mu^{n+1,n}$. Similarly, the second marginal of $\pi^{n,n}$ is fitted, however the first one might not be correct and hereafter we will denote it as  $\rho^{n,n}$. 
Moreover, $\pi^{n+1,n}$ is the optimal entropic plan associated to the EOT problem with marginals $\rho,\,\mu^{n+1,n}$ whereas $\pi^{n,n}$ is the optimal EOT plan associated to the problem with marginals $\rho^{n,n},\,\mu$. 
Due to this partial marginal fitting nature of the algorithm, since we can see Sinkhorn plans $\{\pi^{n+1,n}\}_{n\in\N}$ as a sequence of entropic plans where the first marginal is always fixed and the second one changes according to $\{\mu^{n+1,n}\}$, we see that proving the exponential convergence of the algorithm boils down to apply quantitative stability estimates and to control the sequence of wrong marginals. For these reasons, Sinkhorn's algorithm and quantitative convergence bounds quantitative stability bounds for EOT are two problems tightly related and both have been addressed from a vast literature (see literature review below). Despite this, in the unbounded settings, much less has been known until the recent contributions of \cite{conforti2023Sinkhorn,eckstein2023hilberts,chiarini2024semiconcavity}, where this problem has been addressed in full generality and where exponential convergence rates were shown to hold in relative entropy for Sinkhorn plans and in $\rmL^p$-norm (with $p\in\{1,2\}$) for gradients of Sinkhorn potentials. Here our Hessian stability estimates allow us to deduce also a second-order convergence result, \ie, that the Hessian of Sinkhorn potentials converges exponentially fast with the same rate obtained in \cite{chiarini2024semiconcavity} for Sinkhorn plans. To state it, let us recall that a probability measure $\nu\in\cP(\bbRD)$ is said to satisfy a Talagrand inequality with constant $\tau$, \ref{eq:TI} for short, if 
\be\label{eq:TI}\tag{TI($\tau$)}
\bfW_2^2(\mu,\nu) \leq 2\tau\,\scrH(\mu|\nu)\,, \quad \forall \mu \in\cP(\bbRD)\,.
\ee

\begin{theorem}\label{thm:sinkhorn:hessian}
    Assume \Cref{ass:basic}  and  that there exist $\Lambda\in(0,+\infty)$ and $N\geq 2$ such that the function $g^x_{\varphi_0^n}$ 
is $\Lambda$-semiconcave uniformly in $x\in\supp(\rho)$ and $n\geq N$. If $\mu^{n,n-1}$ satisfies \ref{eq:TI} for some $\tau\in(0,+\infty)$ and for all $n\geq N$, then 

 \bes
    \begin{aligned}
    &\|\nabla\varphi^{n+1}-\nabla\varphi^\mu \|^2_{\rmL^2(\rho)}\lesssim\bigg(1-\frac{  T }{T+\tau \Lambda}\bigg)^{(n-N+1)}\tau\, \scrH(\pi^{\mu}|\pi^{0,0})\,,    \\
    &\|    \nabla^2 \varphi^{{n+1}}-\nabla^2 \varphi^\mu\|_{\rmL^1(\rho)} \lesssim  \bigg(1-\frac{  T }{T+\tau \Lambda}\bigg)^{\frac{n-N+1}{2}}\,\sqrt{\tau\scrH(\pi^{\mu}|\pi^{0,0})}+\bigg(1-\frac{  T }{T+\tau \Lambda}\bigg)^{(n-N+1)} \tau\scrH(\pi^{\mu}|\pi^{0,0}),    
    \end{aligned}
    \ees
hold for all $n\geq N$ up to multiplicative constants that depend polynomially only on $\rho,\,\mu,\,T$ (and not on the iterates). These constants are explicit. 

In particular, up to numerical universal constants, we have
    \begin{itemize}
        \item if $\supp(\rho),\,\supp(\mu)\subseteq B_R(0)$ (for some radius big enough, \ie, $R^2\geq T$), then the uniform semiconcavity parameter reads as $\Lambda=R^2\,T^{-1}$ and 
        \bes 
        \begin{aligned}
            &\|\nabla\varphi^{n+1}-\nabla\varphi^\mu \|^2_{\rmL^2(\rho)} \lesssim\nicefrac{R^4}{T^4}\,\bigg(1-\frac{  T^2 }{T^2+\tau R^2}\bigg)^{(n-N+1)} \tau\,\scrH(\pi^{\mu}|\pi^{0,0})\,,   \\
    &\|    \nabla^2 \varphi^{{n+1}}-\nabla^2 \varphi^\mu\|_{\rmL^1(\rho)} \lesssim
   (\nicefrac{R^4}{T^{\nicefrac72}}+\nicefrac{d}{T})\,\bigg(1-\frac{  T^2 }{T^2+\tau R^2}\bigg)^{\frac{n-N+1}{2}}\,\sqrt{\tau\scrH(\pi^{\mu}|\pi^{0,0})}\\
   &\qquad\qquad\qquad\qquad\qquad\qquad+
   \nicefrac{R^6}{T^5}\,\bigg(1-\frac{  T^2 }{T^2+\tau R^2}\bigg)^{(n-N+1)} \tau\,\scrH(\pi^{\mu}|\pi^{0,0})\,,
   \end{aligned}
    \ees
    \item if both $\rho$ and $\mu$ are log-concave, \ie, their (negative) log-densities satisfy  $
 \nabla^2U_\rho\geq \alpha_\rho$  and $\nabla^2U_\mu\geq \alpha_\mu$ for some  $\alpha_\rho,\,\alpha_\mu>0$ (wlog such that $\alpha_\rho\vee\alpha_\mu<T^{-1}$),  then the uniform semiconcavity parameter reads as $\Lambda=(\alpha_\rho\,T)^{-1}$ and 
   \bes 
        \begin{aligned}
            &\|\nabla\varphi^{n+1}-\nabla\varphi^\mu \|^2_{\rmL^2(\rho)}\lesssim \frac{\tau}{\alpha_\rho\,\alpha_\mu\,T^4}\,\bigg(1-\frac{  \alpha_\rho\,T^2 }{\alpha_\rho\,T^2+\tau }\bigg)^{(n-N+1)} \scrH(\pi^{\mu}|\pi^{0,0})\,, \\
    &\|    \nabla^2 \varphi^{{n+1}}-\nabla^2 \varphi^\mu\|_{\rmL^1(\rho)} \lesssim
   \biggl(\frac1{\alpha_\mu\,\sqrt{\alpha_\rho}\,T^3}+\frac{d}{\sqrt{\alpha_\rho\,\alpha_\mu}\,T^2}\biggr)\,\bigg(1-\frac{  \alpha_\rho\,T^2 }{\alpha_\rho\,T^2+\tau }\bigg)^{\frac{n-N+1}{2}}\,\sqrt{\tau\scrH(\pi^{\mu}|\pi^{0,0})}\\
   &\qquad\qquad\qquad\qquad\qquad\qquad\qquad\qquad\qquad+
   \frac\tau{\alpha_\rho\,\alpha_\mu\,T^4}\,\bigg(1-\frac{  \alpha_\rho\,T^2 }{\alpha_\rho\,T^2+\tau }\bigg)^{(n-N+1)} \scrH(\pi^{\mu}|\pi^{0,0})\,.
   \end{aligned}
    \ees
    \end{itemize}
\end{theorem}

Let us remark that the uniform $\Lambda$-semiconcavity of $g^x_{\varphi^n_0}$ (as defined in~\eqref{eq:def:costo+funzione}) and the Talagrand inequality are the same assumptions considered in \cite{chiarini2024semiconcavity} when proving the exponential convergence of Sinkhorn's plans, and combining them with our stability bounds for gradients and Hessians leads to the above first- and second-order convergence for Sinkhorn's iterates. Moreover, as shown in \Cref{sub:rho:compact} and \Cref{app:rho:logconcave:no:beta}, whenever $\rho$ is either compactly supported or log-concave, we are guaranteed that the function $g^x_{\varphi^n_0}$ is $\Lambda$-semiconcave uniformly in $n\in\N$, whence the validity of the first assumption in the previous theorem.

As concerns the uniform Talagrand inequality assumption on the marginals $\mu^{n,n-1}$ generated along Sinkhorn's algorithm, let us comment its role and a possible alternative assumption. As for the role, we employ Talagrand inequality for a twofold reason: to apply \cite{chiarini2024semiconcavity} and to translate their entropic convergence estimates into our $\bfW_2$-stability bounds (cf.~\eqref{eq:cor:from:sink}). As for the alternative assumption, our convergence result can still be obtained by alternatively assuming $\mu$ to satisfy a Talagrand inequality (in order to apply \cite[Theorem 1.2]{chiarini2024semiconcavity}) and Sinkhorn's plans $\pi^{n,n-1}$ to satisfy a Talagrand inequality (in order to translate entropic bounds into Wasserstein ones). This last condition is met if for instance $\rho$ satisfies a log-Sobolev inequality, which implies its validity also for Sinkhorn's plans $\pi^{n,n-1}$ \cite[Theorem 1.3]{conforti2024weak}, and hence a Talagrand inequality (see for instance \cite[Theorem 9.6.1]{bakry2013analysis}).

\subsection{Schr\"odinger bridge point of view}\label{sec:SB:processes}

Our proof strategy relies on the stochastic control representation of entropic plans as laws of solutions to time-inhomogeneous SDEs. More precisely, we are going to consider the forward Schr\"odinger bridge process (from $\rho$ to $\mu$) defined as the SDE driven by $-\nabla\psi^\mu_s$, that is the stochastic process $(X^{\psi^\mu,\rho})_{s\in[0,T]}$ solution to (cf. \cite{conforti2024weak})
\begin{equation}\label{eq:def:SB:dritto}
\De X^{\psi^\mu,\rho}_s = -\nabla\psi^\mu_s(X^{\psi^\mu,\rho}_s)\De s+\De B_s\,,\quad X_0^{\psi^\mu,\rho}\sim\rho\,.
\end{equation}
Then the joint law $\cL(X^{\psi^\mu,\rho}_0, X^{\psi^\mu,\rho}_{T})$ coincides with the optimal entropic coupling $\pi^\mu$, \ie, the solution to EOT with marginals $\rho,\mu$.

Similarly, we will consider its time-reversal corresponding process, \ie, the (backward) Schr\"odinger bridge (from $\mu$ to $\rho$) which solves
\begin{equation}\label{eq:def:SB:back}
\De X^{\varphi^\mu,\mu}_s = -\nabla\varphi^\mu_s(X^{\varphi^\mu,\mu}_s)\De s+\De B_s\,,\quad X^{\varphi^\mu,\mu}_0\sim\mu\,.
\end{equation}
Let us recall here that the bridge 
$X^{\varphi^\mu,\mu}_\cdot$ is the time-reversal process 
of the forward bridge $X^{\psi^\mu,\rho}_\cdot$, \ie, for any $s\in[0,T]$ the following identities in law hold
\be\label{eq:time:reversal:bridge}
X^{\varphi^\mu,\mu}_s\overset{\text{law}}{=}X^{\psi^\mu,\rho}_{T-s} \,,
\ee
and clearly that $\cL(X^{\varphi^\mu,\mu}_T,X^{\varphi^\mu,\mu}_0)=\cL(X^{\psi^\mu,\rho}_0, X^{\psi^\mu,\rho}_{T})=\pi^\mu$. 

In light of these representations, it is clear that semiconcavity and functional properties of the EOT plan $\pi^\mu$ are affected by convexity properties of the drifts appearing in~\eqref{eq:def:SB:dritto} and~\eqref{eq:def:SB:back}, as already noticed in \cite{conforti2024weak}. For this reason, alongside the semiconcavity parameter $\Lam{\varphi^\mu_0}$ our constants appearing below will depend on lower bounds on the Hessians of propagated potentials along HJB, \ie, for any $h\in\{\varphi^\mu,\,\psi^\mu\}$ we will consider the lower bounds $\nabla^2h_s\geq \lambda(h_s)$ with $\lambda(h_s)\in\rset$. 
Our general results are stated for any given sequence $(\lambda(h_s))_{s\in[0,T)}$ satisfying this lower bound (and we do not assume it to be the optimal one, as we did for $\Lam{\varphi^\mu_0}$). In the Appendix we  provide explicit lower bounds for the examples considered here.

\subsection{Literature review}

\noindent \textbf{Quantitative stability.} In recent years a rich literature has flourished around quantitative stability questions for primal and dual solutions of the EOT problem.

At the level of entropic plans, let us mention \cite{lagg2022gradient} and \cite{eckstein2021quantitative}. In the former, the difference in (symmetric) entropy between the solutions to two different EOT problems is controlled in terms of a negative Sobolev norm, for a wide class of problems with costs induced by diffusions on Riemannian manifolds with Ricci curvature bounded from below (which includes the quadratic cost on $\bbRD$). The latter obtains instead a quantitative H\"older estimate between the Wasserstein distance of optimal plans and that of their marginals.  Let us further cite \cite{ghosal2021stability}, where a more qualitative stability result is proven under mild hypotheses.
Finally, \cite{chiarini2024semiconcavity} provides a control on the entropy between two entropic plans in terms of the (squared) Wasserstein distance between the marginals. The peculiarity of this last work is the approach, since it exploits for the first time the propagation of semiconcavity along HJB to obtain a quantitative stability result for primal solutions. The second-order quantitative stability bounds on entropic and Sinkhorn potentials that we will show in this manuscript build upon this previous contribution. For this reason and for sake of completeness, we prove the entropic stability estimate via semiconcavity also in the present manuscript, but we provide a different proof, based on the stochastic representation of Schr\"odinger bridges (see \Cref{thm:entropic_stab} below).

As concerns dual solutions, \ie, entropic potentials, in \cite{carlier2020differential} an $\rmL^\infty$-Lipschitz bound is obtained; it applies to multimarginal OT, but it requires either the space or the cost to be bounded. In \cite{deligiannidis2021quantitative} the $\rmL^\infty$-norm of the difference between entropic potentials associated to two EOT problems is controlled by the Wasserstein distance of order one between the corresponding marginals, using an approach based on Hilbert's metric; but again, both the cost function and the marginals' supports are assumed to be bounded. On the other hand, \cite{carlier2024displacement} succeed in controlling the same difference with the Wasserstein distance of order two of the respective marginals, provided the cost is $C^{2,\infty}$, \ie, bounded with two bounded derivatives; if the regularity of the cost is higher, say $C^{k+2,\infty}$, then the $\rmL^\infty$-norm of the difference between entropic potentials can be replaced by the $C^{k,\infty}$-norm. The interest in higher-order stability results for entropic potentials is motivated by the fact that their gradients provide good proxies for OT maps (\cite{greco2024thesis,malamut2023convergence,lagg2022gradient} in unbounded settings and \cite{pooladian2021entropic} in semidiscrete ones) and entropic estimates can be leveraged to obtain in the $T$ vanishing limit estimates for Kantorovich potentials and OT maps  \cite{fathi2019proof, chewi2022entropic, kitagawa2025stabilityoptimaltransportmaps}. In particular, very recently \cite{kitagawa2025stabilityoptimaltransportmaps} rely on estimates for regularized potentials combined with gluing arguments in the vanishing $T$ limit, to get quantitative stability estimates for OT maps.
Finally, in \cite{divol2024tight} the $\rmL^2$-norm of the difference of the gradients of entropic potentials is controlled in a Lipschitz way by the Wasserstein distance between the corresponding marginals, using a functional inequality for tilt-stable probability measures, see \cite{chen2022localization} and \cite[Lemma 3.21]{bauerschmidt2024stochastic}, and under the assumption that both entropic potentials have bounded Hessian. The dependence of the Lipschitz constant on the regularization parameter is polynomial, thus improving on earlier results, and marginals may have unbounded support.
Among those just mentioned, this is the closest contribution to ours, since \cite{divol2024tight} use Lipschitzianity of the Schr\"odinger maps (and hence concavity/convexity bounds for entropic potentials) in order to prove stability bounds for $\nabla\varphi^\mu$. Therefore, our work can be seen as an extension to second-order quantitative bounds. Moreover, our stability bounds for the gradients behave as theirs (our potentials and theirs differ from a multiplicative prefactor $-T$). In the compact setting we get the same asymptotic behavior in $R$ and $T$, whereas if we put ourselves in the Caffarelli's setting (\ie, Hessian of marginals upper and lower bounded), then our general estimate would not depend on $T$ and would behave as their stability result when assuming bi-Lipschitzianity.

\medskip

\noindent \textbf{Sinkhorn's algorithm.} Contributions to Sinkhorn's algorithm in the literature date back to \cite{Sinkhorn64} and \cite{SinkhornKnopp67}. It was originally considered in a discrete setting framework for doubly stochastic matrices and the first exponential convergence result was given in \cite{Franklin89hilbert,borwein1994entropy} by exploiting properties of Hilbert's projective metric. After the seminal work of  \cite{cuturi2013sinkhorn}, which opened up to possible application of EOT to  machine learning, multiple papers dealt with the convergence of the algorithm. Particularly in bounded settings (\ie, compact spaces or bounded costs) this has already been well established in \cite{chen2016hilbertmetric,marinogerolin2020,Carlier22multisink}. In particular \cite{chen2016hilbertmetric} obtained the first exponential convergence results in the continuous setting using the Hilbert's metric approach. However, this approach provides rates that depend exponentially on the regularizing parameter $T$ and cannot be extended to unbounded settings. 

On the other hand, much less was known for unbounded settings (including the most iconic and simple quadratic cost setting with log-concave marginals). In fact, the only widely general known qualitative convergence result was due to \cite{ruschendorf1995convergence}, and it has been recently improved in \cite{nutz2021entropic}. The first quantitative result in unbounded settings we are aware of is \cite{eckstein2021quantitative}, subsequently improved in \cite{ghosal2022nutz}, where the authors prove a polynomial convergence rate. These works are based on  quantitative stability estimates for EOT and the insightful interpretation of Sinkhorn's algorithm as a block-coordinate descent algorithm on the dual problem \cite{leger2021gradient,aubin2022mirror,leger2023gradient}.

Only very recently, it has been established the exponential convergence for unbounded costs and marginals. Up to the authors' knowledge, the first contribution in this setting is given in  \cite{conforti2023Sinkhorn}, which studies the quadratic cost. There, the main result is that if the marginals are weakly log-concave and the regularization parameter $T$ is large enough, exponential convergence of the gradients of the iterates holds (and their results work for any $T>0$ for Gaussian marginals). 
Moreover, this is the first contribution that has highlighted how geometric assumptions on the marginals (such as log-concavity) can be leveraged to improve the dependence in the convergence rates, from exponential to polynomial in $T$. Later, following similar considerations, \cite{chizat2025sharper} improved the exponential convergence results in the bounded setting, showing that the exponential rate of convergence deteriorates polynomially in $T$. With regard to the unbounded setting, \cite{eckstein2023hilberts} has subsequently managed to construct a suitable version of Hilbert's metric for general unbounded costs. In contrast with \cite{conforti2023Sinkhorn}, exponential convergence is shown for all values of $T$, under a growth condition assumption. Roughly speaking, therein the author assumes that the tails of the marginals decay (strictly) faster than the cost function considered.  When applied to the quadratic cost, this assumption does not completely cover log-concave distributions and their perturbations, leaving out Gaussian marginals for example. 

The first paper that has finally managed to provide exponential convergence rates in general (possibly unbounded) settings, working for any regularization parameter $T>0$ and with polynomial dependence in $T$, is~\cite{chiarini2024semiconcavity}. There, together with our coauthors, we show how  semiconcavity bounds on Sinkhorn potentials can be leveraged to obtain exponential convergence. Our geometric approach is broadly general and covers as particular cases the bounded settings as well as the (anisotropic) quadratic costs, which include, for instance, also the case when the cost function is the transition kernel induced by an Ornstein--Uhlenbeck process (\ie, the framework of the Schr\"odinger bridge problem with a non-Gaussian reference process). The key observation employed there is that the semiconcavity of the function defined in~\eqref{eq:def:costo+funzione} is enough to deduce quantitative stability estimates and exponential convergence rates depending on the semiconcavity parameter $\Lam{\varphi^\mu_0}$.

Lastly, it is worth mentioning different contributions that
over the past few years have focused on different asymptotic properties of Sinkhorn's algorithm. Let us just mention \cite{berman2020sinkhorn} for the relation with Monge--Amp\`ere equation, \cite{deb2023wasserstein} for the construction of Wasserstein mirror gradient flows,  \cite{sander2022sinkformers} for construction of a Transformer variant inspired by Sinkhorn's algorithm, and the very recent series of contributions \cite{akyildiz2024gaussianentropicoptimaltransport,akyildiz2025contractionpropertiessinkhornsemigroups,delmoral2025stabilityschrodingerbridgessinkhorn} that focus on the relation of Schr\"odinger bridges and Sinkhorn's algorithm with the Riccati matrix difference equations, and the impact of these results in the context of multivariate linear Gaussian models and statistical finite mixture models. We conclude by mentioning the recent work \cite{eckstein2025exponentialconvergencegeneraliterative}, where the authors investigate the convergence of IPFP for a more general class of problems (which includes EOT), whose proof is based on strong convexity arguments for the dual problem, highlighting the role of the geometric interplay between the subspaces defining the constraints.

\medskip

We would like to conclude this review by mentioning results that have inspired us or are related to ours.  As we have already stated, our strategy is based on stochastic analysis and second-order estimates along HJB equations. This approach has been initially introduced in \cite{conforti2024weak}, where Conforti has proved weak semiconcavity estimates for entropic potentials by studying how this property propagates along HJB equations. In  \cite{conforti2023Sinkhorn} and \cite{chiarini2024semiconcavity}, this has already been employed for proving the exponential convergence of Sinkhorn's algorithm and for showing stability estimates of entropic plans. Here we further extend its use to show second-order stability estimates. 
In order to prove the convergence of Sinkhorn's algorithm, a similar approach has been employed also in \cite{greco2023SinkhornTorus,greco2024thesis} where we have studied how Lipschitzianity propagates along HJB equations, leading to a more perturbative convergence result (instead of a geometric one).
Lastly, we would like to mention \cite{chaintron2025regularitystabilitygibbsconditioning}, though not directly applied to EOT; there, the authors provide third-order estimates propagated along HJB in order to prove stability estimates for stochastic optimal control problems. These new ideas open up to further investigation of third-order estimates for entropic potentials.

%%%%%%% below proofs %%%%%%%%%

\section{Preliminaries}\label{sec:preliminari}

In this paper we are interested in the behavior of the forward process $(Y^\theta_s)_{s\in[0,T]}$ and backward process $(Y^\eta_s)_{s\in[0,T]}$ defined as $Y^\theta_s\coloneqq \nabla\theta_s(X^{\psi^\mu,\rho}_s)$ where $\theta_s\coloneqq \psi^\nu_s-\psi^\mu_s$ and similarly as $Y^\eta\coloneqq \nabla\eta_s(X^{\varphi^\mu,\mu}_s)$ with $\eta_s\coloneqq \varphi^\nu_s-\varphi^\mu_s$. Since both $\varphi^\nu_\cdot$ and $\varphi^\mu_\cdot$ solve~\eqref{eq:HJB} it is immediate to see that $\eta_\cdot$ and $\theta_\cdot$ solve
\bes
\partial_s \eta_s + \frac12\Delta \eta_s -\nabla \varphi^\mu_s\cdot \nabla \eta_s -\frac12|\nabla\eta_s|^2=0 \,, \qquad \partial_s \theta_s + \frac12\Delta \theta_s -\nabla \psi^\mu_s\cdot \nabla \theta_s -\frac12|\nabla\theta_s|^2=0 \,.
\ees
Hence from It\^o's formula we further deduce that 
\begin{subequations}
\be\label{eq:sde:Yeta}
\De Y^\eta_s=\nabla^2\varphi^\nu_s(X^{\varphi^\mu,\mu}_s)Y^\eta_s\,\De s+\nabla^2\eta_s(X^{\varphi^\mu,\mu}_s)\,\De B_s\,,
\ee
\be\label{eq:sde:Ytheta}
\De Y^\theta_s=\nabla^2\psi^\nu_s(X^{\psi^\mu,\rho}_s)Y^\theta_s\,\De s+\nabla^2\theta_s(X^{\psi^\mu,\rho}_s)\,\De B_s\,.
\ee
\end{subequations}
Finally, notice that $
\bbE[|Y^\eta_T|^2]=\|\nabla\varphi^\nu-\nabla\varphi^\mu \|_{\rmL^2(\rho)}^2 =\|\nabla\psi^\nu_0-\nabla\psi^\mu_0\|_{\rmL^2(\rho)}^2=\bbE[|Y^\theta_0|^2]$.

Besides the relation at initial and terminal times with the integrated difference between the gradients of the potentials, the processes $(Y^\eta_s)_{s\in[0,T]}$ and $(Y^\theta_s)_{s\in[0,T]}$ play a crucial role since their integrated in time mean squares measure the entropic distance between $\pi^\nu$ and $\pi^\mu$. Namely, from Girsanov's theory we know that 
\bes
\frac12\,\int_0^{\delta T}\bbE[|Y^\eta_s|^2]\De s=\scrH(\cL(X^{\varphi^\mu,\mu}_{[0,\delta T]})|\cL(X^{\varphi^\nu,\mu}_{[0,\delta T]}))
\ees
and in particular whenever $\mu\ll\nu$ and for $\delta=1$ we then have
\be\label{eq:energia:entropia:running}
\frac12\,\int_0^{ T}\bbE[|Y^\eta_s|^2]\De s=\bbE_\mu[\scrH(\pi^{\mu}(\cdot|X)|\pi^\nu(\cdot|X))] \,,
\ee
which gives rise to 
\be\label{eq:energia:entropia}
\scrH(\mu|\nu)+\frac12\,\int_0^ T\bbE[|Y^\eta_s|^2]\De s=\scrH(\cL(X^{\varphi^\mu,\mu}_{[0, T]})|\cL(X^{\varphi^\nu,\nu}_{[0, T]}))=\scrH(\pi^\mu|\pi^\nu)\,
\ee
whenever $\scrH(\mu|\nu)$ is finite.
Similarly we have 
\be\label{eq:energia:entropia:forward}
\frac12\,\int_0^ {\delta T}\bbE[|Y^\theta_s|^2]\De s=\scrH(\cL(X^{\psi^\mu,\rho}_{[0, \delta T]})|\cL(X^{\psi^\nu,\rho}_{[0, \delta T]}))\,,
\ee
which equals $\scrH(\pi^\mu|\pi^\nu)$ for $\delta=1$.

\subsection{Entropic and gradients' stability}

Let us start by showing how the stochastic control point of view can be employed in studying entropic stability of the optimal plans.

\begin{theorem}\label{thm:entropic_stab}
Assume \Cref{ass:basic} and let $\pi^\nu,\,\pi^\mu$ denote the optimal plans associated to EOT with marginals $(\rho,\nu)$ and $(\rho,\mu)$ respectively. Then the following entropic stability bound holds
\be\label{eq:ent_stab}
\scrH(\pi^{\mu}|\pi^\nu)\leq \scrH(\mu|\nu)+\frac{\Lam{\varphi^\nu_0}}{2T}\,\bfW_2^2(\mu,\nu)\,.
\ee
Moreover, if \Cref{ass:mu} holds then we have 
\begin{subequations}
\be\label{eq:bound:conditionals:W2}
\bbE_\mu[\scrH(\pi^{\mu}(\cdot|X)|\pi^\nu(\cdot|X))]\leq \frac{\Lam{\varphi^\nu_0}}{2T}\,\bfW_2^2(\mu,\nu) \,,
\ee
\begin{equation}\label{eq:bound:Y:eta:0}
\bbE[|Y^\eta_0|^2]=\| \nabla\varphi^\mu_0-\nabla\varphi^\nu_0 \|_{\rmL^2(\mu)}^2 \leq \frac{\Lam{\varphi^\nu_0}\,\Cphi}{T^2}\,\bfW_2^2(\mu,\nu)\,,
\end{equation}
\end{subequations}
where the positive constant $\Cphi$ is defined as  
  \be\label{eq:def:Cphi}
  \Cphi\coloneqq T\biggl (\int_0^{T}e^{\int_0^s 2\lambda(\varphi^\nu_t)\De t}\De s\biggr)^{-1} \,.
  \ee
\end{theorem}

The bound \eqref{eq:ent_stab} has already been proven by the authors and collaborators in \cite[Theorem 1.1]{chiarini2024semiconcavity}. We report it here since its proof can be employed in order to get \eqref{eq:bound:conditionals:W2} and \eqref{eq:bound:Y:eta:0}, which will play a crucial role in the rest of the paper. Moreover, let us also remark that we provide here a stochastic analysis proof of \eqref{eq:ent_stab} by building a suitable competitor using a modified Schr\"odinger bridge process (see \eqref{eq:modified:SB} below). Before proving this result, let us state
a technical bound, akin to \cite[Lemma 2.1]{chiarini2024semiconcavity}. In particular the following lemma can be seen as a generalization of it, where we bound the relative entropy between $\pi^\mu_{s}(\cdot|y)=\cL(X^{\varphi^\mu,y}_s)$ and $\pi^\nu_s(\cdot|z)=\cL(X^{\varphi^\nu,z}_s)$, where we recall from~\eqref{eq:def:SB:back} the backward Schr\"odinger bridge process (started in $x\in\bbRD$) being defined by
\be\label{eq:modified:SB}
\De X^{\varphi^\nu,x}_s = -\nabla\varphi^\nu_s(X^{\varphi^\nu,x}_s)\De s+\De B_s\,,\quad X^{\varphi^\nu,x}_0=x.
\ee

\begin{lemma}\label{lem:ent_cond_prob_diff_bridges}
Assume \Cref{ass:basic}. For any $s\in(0,T]$ and any $y\in\supp(\mu)$ and $z\in\supp(\nu)$ it holds
\bes
	\scrH(\pi^\mu_s(\cdot|y)|\pi^\nu_s(\cdot|z)) \leq\frac{\Lambda(\varphi^\nu_0)}{2T}|z-y|^2+(s^{-1}-T^{-1})\frac{|z-y|^2}{2}+\bbE[\eta_s(X^{\varphi_\mu,y}_s)-\eta_0(y)]-\langle\nabla\eta_0(y),z-y\rangle\,.
\ees
\end{lemma}
\begin{proof}
Firstly, observe that the conditional probability measure $\pi^\mu_s(\cdot|y)$  admits a density of the form
\bes
\pi^\mu_s(\De x|y) =(2\pi s)^{-\nicefrac{d}{2}}\, \exp\biggl(-\varphi^\mu_s(x)+\varphi^\mu_0(y)-\frac{|x-y|^2}{2s}\biggr)\De x\,,
\ees
and a similar expression holds for $\pi^\nu_s(\cdot|z)$. Therefore we may rewrite the relative entropy as
\bes
	\begin{aligned}
	\scrH(\pi^\mu_s(\cdot|y)|\pi^\nu_s(\cdot|z)) = \varphi^\mu_0(y)-\varphi^\nu_0(z)+\int (\varphi^\nu_s-\varphi^\mu_s)(x) +\frac{|x-z|^2-|x-y|^2}{2s}\pi^\mu_s(\De x|y)\\
 =\varphi^\mu_0(y)-\varphi^\nu_0(z)+\frac{|z|^2-|y|^2}{2s}+\int \eta_s(x)+s^{-1}\langle x,y-z\rangle\,\pi^\mu_s(\De x|y) \\
 = \varphi^\mu_0(y)-\varphi^\nu_0(z)+\frac{|z|^2-|y|^2}{2s}+\bbE[\eta_s(X^{\varphi_\mu,y}_s)+s^{-1}\langle X^{\varphi^\mu,y}_s, y-z\rangle]\,.
	\end{aligned}
\ees
Next, since $(\nabla\varphi^\nu_s(X^{\varphi^\nu,y}_s))_{s\in[0,T]}$ is a martingale (cf.\ \cite[Proof of Theorem 2.1]{conforti2024weak}, namely it follows from It\^o's formula combined with the Hamilton--Jacobi--Bellman equation and the SDE of the Schr\"odinger bridge), we have
 \bes
 \langle  \bbE[\nabla\varphi^\nu_s(X^{\varphi^\nu,y}_s)],y-z\rangle = \langle\nabla\varphi^\nu_0(y),y-z\rangle \,,
\ees
so that if we integrate from $0$ to $s$ the dynamics of $X_s^{\varphi^\nu,y}$ and take expectations, we get
\begin{equation*}
\begin{aligned}
\bbE[\langle X^{\varphi^\nu,y}_s,y-z\rangle]= & \bbE[\langle X^{\varphi^\nu,y}_0,y-z\rangle]-\int_0^s \langle\bbE[\nabla\varphi^\nu_t(X^{\varphi^\nu,y}_t)],y-z\rangle] \,\De t \\=&
\langle y,y-z\rangle +s \, \langle\nabla\varphi^\nu_0(y),z-y\rangle\,.
\end{aligned}
\end{equation*}
Hence we conclude that
\bes
\begin{aligned}
\scrH(\pi^\mu_s(\cdot|y)|\pi^\nu_s(\cdot|z)) = \varphi^\mu_0(y)-\varphi^\nu_0(z)+\frac{|z-y|^2}{2s}+\bbE[\eta_s(X^{\varphi_\mu,y}_s)]+ \langle\nabla\varphi^\mu_0(y),z-y\rangle\\
 = \varphi^\nu_0(y)-\varphi^\nu_0(z)+\langle\nabla\varphi^\nu_0(y),z-y\rangle+\frac{|z-y|^2}{2s}+\bbE[\eta_s(X^{\varphi_\mu,y}_s)-\eta_0(y)] \\-\langle\nabla\eta_0(y),z-y\rangle\\
 \leq\frac{\Lambda(\varphi^\nu_0)}{2T}|z-y|^2+(s^{-1}-T^{-1})\frac{|z-y|^2}{2}+\bbE[\eta_s(X^{\varphi_\mu,y}_s)-\eta_0(y)]-\langle\nabla\eta_0(y),z-y\rangle\,,
\end{aligned}
\ees
where in the last step we have noticed that
\bes
T\biggl( \varphi^\nu_0(y)-\varphi^\nu_0(z)+\langle\nabla\varphi^\nu_0(y),z-y\rangle+\frac{|z-y|^2}{2T}\biggr)=g^y_{\varphi^\nu_0}(z)-g^y_{\varphi^\nu_0}(y)-\langle\nabla g^y_{\varphi^\nu_0}(y), z-y\rangle\,,
\ees
with $g^y_{\varphi^\nu_0}(z)=\frac{|z-y|^2}{2}-T\,\varphi^\nu_0(z)$, and we have used its $\Lam{\varphi^\nu_0}$-semiconcavity.
\end{proof}

In the particular case $s=T$ and $\mu=\nu$ (henceforth $\eta_\cdot=0$), the above result simply reads as \cite[Lemma 2.1]{chiarini2024semiconcavity}, that is

\begin{corollary}\label{cor:entropy:shifted:conditionals}
Assume \Cref{ass:basic}. For any $y,z\in\supp(\nu)$ we have $\scrH(\pi^\nu(\cdot|y)|\pi^\nu(\cdot|z)) \leq \frac{\Lam{\varphi^\nu_0}}{2T}\,|z-y|^2$.
\end{corollary}

\begin{proof}[Proof of \Cref{thm:entropic_stab}]
Let us focus on \eqref{eq:ent_stab} first. Without loss of generality we may assume  $\scrH(\mu|\nu),\,\Lam{\varphi^\nu_0},\,\bfW_2^2(\mu,\nu)$ to be all finite, otherwise there is nothing to prove.
Next, observe that $\pi^{\mu}$ can be seen as the entropic optimal plan w.r.t.\ the reference measure $\pi^\nu$ for the EOT problem
\be\label{eq:SB_relative} 
\scrH(\pi^{\mu}|\pi^\nu)= \min_{\pi\in\Pi(\rho,\mu)}\scrH(\pi|\pi^\nu) \,.
\ee
This directly follows from \cite[Theorem 2.1.b]{Marcel:notes} after noticing that almost surely $
\frac{\De \pi^{\mu}}{\De \pi^\nu} = \exp((\varphi^\nu-\varphi^{\mu})\oplus(\psi^\nu-\psi^{\mu}))$
and hence also $\pi^\nu$-a.s.\,  Notice that \cite[Theorem 2.1.b]{Marcel:notes} further implies $\scrH(\pi^{\mu}|\pi^\nu)<\infty$. 
We now proceed to bound $\scrH(\pi^{\mu}|\pi^\nu)$ exhibiting a suitable admissible plan in~\eqref{eq:SB_relative}.  
In view of that, let us consider the optimal transport map $\cT$ from $\mu$ to $\nu$, that is such that $\cT_{\#}\mu=\nu$ (see~\cite[Theorem 1.22]{Santambrogio2015} for the existence of the optimal transport map for $\bfW_2(\mu,\nu)$). Next, for any $x\in\bbRD$ consider the backward process $(X^{\varphi^\nu,x}_s)_{s\in[0,T]}$ defined by~\eqref{eq:modified:SB}, and independently take $X_0\sim\mu$  and define $X_0^\nu\coloneqq \cT(X_0)\sim\nu$. 
Finally, let $\gamma_s=(1-s/T)X_0+(s/T)X^{\nu}_0$ and consider now the stochastic process $X_\cdot$ defined for any $s\in[0,T]$ as $
X_s=X^{\varphi^\nu,\gamma_s}_s$,
and note that if we call $\pi_{\mathrm{comp}}$ the law of $(X_T,X_0)$, then $\pi_{\mathrm{comp}}\in\Pi(\rho,\,\mu)$. 

Then, by optimality of $\pi^\mu$ in \eqref{eq:SB_relative} and by considering $\pi_{\mathrm{comp}}$ as a competitor we may deduce that
\be\label{eq:entropy_bound_disintegrated}
\scrH(\pi^{\mu}|\pi^\nu)\leq \scrH(\pi_{\mathrm{comp}}|\pi^{\nu}) =\scrH(\mu|\nu)+\int \scrH\biggl(\pi_{\mathrm{comp}}(\cdot|z)|\pi^\nu(\cdot|z)\biggr)\,\De \mu(z).
\ee
Next, notice that $X_T=X_T^{\varphi^\nu,X_0^\nu}=X_T^{\varphi^\nu,\cT(X_0)}$ while $X_0=X_0^{\varphi^\nu,X_0}=X_0$, and hence the conditional probabilities appearing in the last display are translations, that is
\bes
\pi_{\mathrm{comp}}(\cdot|z)=\cL(X_T|X_0=z)=\cL(X_T^{\varphi^\nu,\cT(X_0)}|X_0=z)=\cL(X_T^{\varphi^\nu,\cT(z)})=\pi^\nu(\cdot|\cT(z))\,.
\ees
This combined with \eqref{eq:entropy_bound_disintegrated} and \Cref{cor:entropy:shifted:conditionals} proves our claim since the latter implies
\bes\begin{aligned}
\int\scrH(\pi^\nu(\cdot|\cT(z))|\pi^\nu(\cdot|z))\De\mu(z)\leq \frac{\Lam{\varphi^\nu_0}}{2T}\int|\cT(z)-z|^2\De\mu(z) =\frac{\Lam{\varphi^\nu_0}}{2T}\,\bfW_2^2(\mu,\nu)\,.
\end{aligned}
\ees

Let us now focus on the proof of~\eqref{eq:bound:conditionals:W2}.
If in  \Cref{ass:mu} we assume $\scrH(\mu|\nu)<+\infty$, then the conclusion follows from the disintegration property of the relative entropy (see for instance \cite[Lemma 1.6]{Marcel:notes} and \cite[Appendix A]{LeoSch}) since $
\bbE_\mu[\scrH(\pi^{\mu}(\cdot|X)|\pi^\nu(\cdot|X))]= \scrH(\pi^\mu|\pi^\nu)-\scrH(\mu|\nu)$,
which combined with the above entropic stability bound concludes the proof of~\eqref{eq:bound:conditionals:W2} under a finite entropy assumption.

On the other hand, if we assume that $\mu\ll\nu$ with $\Lam{\varphi^\mu_0}$ finite (e.g.\ $\rho$ with compact support or log-concave density), then we can use an approximation argument and consider the sequence of probability measures $\mu^n\in\cP(\bbRD)$ whose densities are defined as 
\bes
\frac{\De\mu^n}{\De \nu}=C_n^{-1}\,\biggl(\frac{\De \mu}{\De \nu}\wedge n\biggr)\,,\quad\text{ with }C_n=\int\biggl(\frac{\De \mu}{\De \nu}\wedge n\biggr)\De\nu\,.
\ees
Clearly, $\mu^n$ converges in $\bfW_2$-distance towards $\mu$ and $C_n\uparrow 1$. Moreover, notice that
\bes
\scrH(\mu^n|\nu)\leq \log(n)-\log(C_n)<+\infty\,,
\ees
and that
\bes
\begin{aligned}
\scrH(\mu^n|\mu)=-\log(C_n)+\int\log\biggl(\IND_{\{\nicefrac{\De\mu}{\De\nu}\leq n\}}+n\,\IND_{\{\nicefrac{\De\mu}{\De\nu}> n\}}\frac{\De\nu}{\De\mu}\biggr)\De\mu^n \\
\leq -\log(C_n)+\int n\,\IND_{\{\nicefrac{\De\mu}{\De\nu}> n\}}\frac{\De\nu}{\De\mu}\,\De\mu^n\leq 1-\log(C_n)<+\infty\,.
\end{aligned}
\ees
 As a first consequence of this, we may deduce from the finite entropy case that 
 \bes\begin{aligned}
 \bbE_\nu\biggl[\frac{\De\mu^n}{\De\nu}(X)\,\scrH(\pi^{\mu^n}(\cdot|X)|\pi^\mu(\cdot|X))\biggr]=
\bbE_{\mu^n}[\scrH(\pi^{\mu^n}(\cdot|X)|\pi^\mu(\cdot|X))]\leq \frac{\Lam{\varphi^\mu_0}}{2T}\,\bfW_2^2(\mu,\mu^n),\,
 \end{aligned}\ees
which vanishes as $n$ diverges. Therefore 
\bes
\lim_{n\to\infty} \biggl(\frac{\De\mu}{\De\nu}(X)\wedge n\biggr)\,\scrH(\pi^{\mu^n}(\cdot|X)|\pi^\mu(\cdot|X))=0\quad\nu-a.s.
\ees
and a fortiori also
\bes
\lim_{n\to\infty} \scrH(\pi^{\mu^n}(\cdot|X)|\pi^\mu(\cdot|X))=0\quad\mu-a.s.\,.
\ees
This implies that $\mu$-a.s. $\pi^{\mu^n}(\cdot|X)$ converges to $\pi^\mu(\cdot|X)$ in total variation (via Pinsker's inequality), henceforth also weakly. From the lower semicontinuity of relative entropy we then deduce that
\bes
\scrH(\pi^{\mu}(\cdot|X)|\pi^\nu(\cdot|X))\leq \liminf_{n\to\infty} \biggl(\frac{\De\mu^n}{\De\mu}\,\scrH(\pi^{\mu^n}(\cdot|X)|\pi^\nu(\cdot|X))\biggr)\quad\mu-a.s.\,.
\ees
By combining this last bound with Fatou's lemma and with the entropic stability estimate already proven above we finally get
\bes
\begin{aligned}
\bbE_\mu[\scrH(\pi^{\mu}(\cdot|X) & |\pi^\nu(\cdot|X))] \leq \liminf_{n\to\infty}\bbE_\mu\biggl[\frac{\De\mu^n}{\De\mu}\,\scrH(\pi^{\mu^n}(\cdot|X)|\pi^\nu(\cdot|X))\biggr] \\
& = \liminf_{n\to\infty}\, \bbE_{\mu^n}[\scrH(\pi^{\mu^n}(\cdot|X)|\pi^\nu(\cdot|X))]
=\liminf_{n\to\infty}\,\scrH(\pi^{\mu^n}|\pi^\nu)-\scrH(\mu^n|\nu) \\
& \leq \liminf_{n\to\infty}\,\frac{\Lam{\varphi^\nu_0}}{2T}\,\bfW_2^2(\mu^n,\nu)=\frac{\Lam{\varphi^\nu_0}}{2T}\,\bfW_2^2(\mu,\nu)\,.
\end{aligned}
\ees

\smallskip

Finally, the proof of~\eqref{eq:bound:Y:eta:0} follows from~\eqref{eq:bound:conditionals:W2} since from It\^o's formula and \eqref{eq:sde:Yeta} we see that 
 \bes
\De \bbE[|Y^\eta_s|^2]\geq 2\,\bbE[Y^\eta_s\cdot \nabla^2\varphi^\nu_s(X^{\varphi^\mu,\mu}_s)Y^\eta_s ]\De s \geq 2\lambda(\varphi^\nu_s)\,\bbE[|Y^\eta_s|^2]\De s\,,
 \ees
 which combined with Gr\"onwall's lemma gives $
\bbE[|Y^\eta_0|^2]\, e^{\int_0^s 2\lambda(\varphi^\nu_t)\De t}\leq \bbE[|Y^\eta_s|^2]$. 
When integrated over $s\in[0, T]$, this inequality reads as
\bes
\bbE[|Y^\eta_0|^2]\leq \biggl(\int_0^{ T}e^{\int_0^s 2\lambda(\varphi^\nu_t)\De t}\De s\biggr)^{-1} \int_0^{ T}\bbE[|Y^\eta_s|^2]\De s=\frac{\Cphi}{T}\int_0^{ T}\bbE[|Y^\eta_s|^2]\De s\,.
\ees
Our thesis then follows by combining this last bound with \eqref{eq:energia:entropia:running} and \eqref{eq:bound:conditionals:W2}.
\end{proof}

From the bound for conditional relative entropies proven in \Cref{cor:entropy:shifted:conditionals} and the gradients' stability bound in \Cref{thm:entropic_stab}, we may deduce an entropic stability bound between $\pi^\nu_s$ and $\pi^\mu_s$.

\begin{corollary}\label{cor:bound_marg_entropy}
Assume \Cref{ass:basic} and \Cref{ass:mu}. Let $\Cphi>0$ as defined in \eqref{eq:def:Cphi}, then we have
\bes
\scrH(\pi^\mu_s|\pi^\nu_s) \leq\biggl(\frac{\Lambda(\varphi^\nu_0)}{T}+\frac{s^{-1}-T^{-1}}{2}+\frac{\sqrt{\Lam{\varphi^\nu_0}\,\Cphi}}{T}\biggr)\bfW_2^2(\mu,\nu)\,.
 \ees
\end{corollary}

\begin{proof}
Firstly, let $\tau\in\Pi(\mu,\nu)$ be the optimal transport coupling between our two target marginals and consider the probability measures on $(\bbRD)^3$ defined by the densities $\pi^\mu_s(\De x|y)\tau(\De y,\De z)$ and $\pi^\nu_s(\De x|z)\tau(\De y,\De z)$ (for notations' sake we indicate these two probabilities respectively with $\pi^\mu_s(\cdot|y)\otimes \tau$ and $\pi^\nu_s(\cdot|z)\otimes\tau$). 

Clearly we have 
\bes
\pi^\mu_s(\De x)=\int\int \pi^\mu_s(\De x|y) \tau(\De y,\De z) \quad\text{and}\quad \pi^\nu_s(\De x)=\int\int \pi^\nu_s(\De x|z) \tau(\De y,\De z)\,,
\ees
therefore, from the data processing inequality and from the disintegration property of relative entropy (cf.\ \cite[Lemma 1.6]{Marcel:notes} and \cite[Appendix A]{LeoSch}) we deduce that
\bes
\scrH(\pi^\mu_s|\pi^\nu_s)\leq \scrH(\pi^\mu_s(\cdot|y)\otimes \tau|\pi^\nu_s(\cdot|z)\otimes\tau)=\int \scrH(\pi^\mu_s(\cdot|y)|\pi^\nu_s(\cdot|z))\,\tau(\De y,\De z)\,.
\ees
Recalling the upper bound given in \Cref{lem:ent_cond_prob_diff_bridges} we get
\bes
\begin{aligned}
\scrH(\pi^\mu_s|\pi^\nu_s) 
\leq \biggl(\frac{\Lambda(\varphi^\nu_0)}{2T}+\frac{s^{-1}-T^{-1}}{2}\biggr)\bfW_2^2(\mu,\nu)+\bbE[\eta_s(X^{\varphi_\mu,\mu}_s)-\eta_0(X_0^{\varphi^\mu,\mu})]+\|\nabla\eta_0\|_{\rmL^2(\mu)}\,\bfW_2(\mu,\nu)\,.
\end{aligned}
\ees
 Now, since $\|\nabla\eta_0\|_{\rmL^2(\mu)}^2=\bbE[|Y^\eta_0|^2]$, from~\eqref{eq:bound:Y:eta:0} we deduce that
\bes\begin{aligned}
\scrH(\pi^\mu_s|\pi^\nu_s) \leq\biggl(\frac{\Lambda(\varphi^\nu_0)}{2T}+\frac{s^{-1}-T^{-1}}{2}+\frac{\sqrt{\Lam{\varphi^\nu_0}\,\Cphi}}{T}\biggr)\bfW_2^2(\mu,\nu)+\bbE[\eta_s(X^{\varphi_\mu,\mu}_s)-\eta_0(X_0^{\varphi^\mu,\mu})]\,.
\end{aligned}\ees
In order to conclude, it is enough noticing that from It\^o's formula it follows
\bes
\De \eta_s(X^{\varphi^\mu,\mu}_s)=\frac12|\nabla\eta_s(X^{\varphi^\mu,\mu}_s)|^2\De s+ \nabla\eta_s(X^{\varphi^\mu,\mu}_s)\De B_s\,,
\ees
which combined with \eqref{eq:energia:entropia:running} and \eqref{eq:bound:conditionals:W2} finally gives 
\bes
\bbE[\eta_s(X^{\varphi_\mu,\mu}_s)-\eta_0(X^{\varphi^\mu,\mu}_0)]=\frac12\int_0^s\bbE[|Y^\eta_t|^2]\De t\leq \bbE_\mu[\scrH(\pi^\mu(\cdot|X)|\pi^\nu(\cdot|X))] \leq\frac{\Lam{\varphi^\nu_0}}{2T}\,\bfW_2^2(\mu,\nu)\,.
\ees
\end{proof}

We conclude this section with one last technical bound.

\begin{proposition}\label{prop:integral:Y}
Assume \Cref{ass:basic} and \Cref{ass:mu}. For any fixed $ \delta\in[0,1)$ we have
\bes
\int_{0}^{\delta T}\bbE[|Y^\theta_s|^2]\De s\leq\biggl(\frac{3\,\Lambda(\varphi^\nu_0)}{T}+\frac{\delta}{1-\delta}\,\frac1T+\frac{2\,\sqrt{\Lam{\varphi^\nu_0}\,\Cphi}}{T}\biggr)\,\bfW_2^2(\mu,\nu)\,.
 \ees
\end{proposition}
\begin{proof}
In view of Girsanov's Theorem identity~\eqref{eq:energia:entropia:forward}, it is enough to notice that
    \bes\begin{aligned}
    2\scrH(\cL(X^{\psi^\mu,\rho}_{[0,\delta T]})|\cL(X^{\psi^\nu,\rho}_{[0,\delta T]}))
   \leq \int_{0}^{ T}\bbE[|Y^\eta_s|^2]\De s+2\scrH(\cL(X^{\varphi^\mu,\mu}_{(1-\delta) T})|\cL(X^{\varphi^\nu,\nu}_{(1-\delta) T}))\\
\overset{\eqref{eq:energia:entropia:running}}{=}2\bbE_\mu[\scrH(\pi^\mu(\cdot|X)|\pi^\nu(\cdot|X))]+2\scrH(\cL(X^{\varphi^\mu,\mu}_{(1-\delta) T})|\cL(X^{\varphi^\nu,\nu}_{(1-\delta) T}))
    \end{aligned}\ees
    where we have relied on a second application of Girsanov's Theorem (as we did for \eqref{eq:energia:entropia}), combined with the time-reversal identities~\eqref{eq:time:reversal:bridge}. Applying \Cref{thm:entropic_stab} and \Cref{cor:bound_marg_entropy} we  concludes our proof.
\end{proof}

\section{Proofs of the main results}

Given the preliminary results of the previous section we are now ready to prove our quantitative stability estimates for gradient and Hessian of the entropic potentials.

\begin{theorem}\label{thm:stability:mappe:general}
Assume \Cref{ass:basic} and \Cref{ass:mu}, fix $\delta\in(0,1)$, and let $
\Cpsi_\delta\coloneqq T\biggl (\int_0^{\delta T}e^{\int_0^s 2\lambda(\psi^\nu_t)\De t}\De s\biggr)^{-1}$. 
Then we have 
\bes
\|\nabla\varphi^\nu-\nabla\varphi^\mu \|_{\rmL^2(\rho)}^2\leq \frac{\Cpsi_\delta}{T}\biggl(2\,\scrH(\pi^\mu_{(1-\delta) T}|\pi^\nu_{(1-\delta )T})+\int_{(1-\delta)T}^T\bbE[|Y^\eta_s|^2]\De s\biggr). 
\ees
As a corollary, if we define $
\Crhonu^\delta\coloneqq\Cpsi_\delta\,\biggl(\frac{\delta}{1-\delta}+3\,\Lam{\varphi^\nu_0}+2\sqrt{\Lam{\varphi^\nu_0}\,\Cphi}\biggr)$,
then we have 
\bes
\|\nabla\varphi^\nu-\nabla\varphi^\mu \|^2_{\rmL^2(\rho)} \leq \frac{\Crhonu^\delta}{T^2}\,\bfW_2^2(\mu,\nu)\,.
\ees
\end{theorem}

 \begin{proof}
 From It\^o's formula  and \eqref{eq:sde:Ytheta}, for all $s\leq \delta T$  we have
 \bes
\De \bbE[|Y^\theta_s|^2]\geq 2\,\bbE[Y^\theta_s\cdot \nabla^2\psi^\nu_s(X^{\psi^\mu,\rho}_s)Y^\theta_s ]\De s \geq 2\lambda(\psi^\nu_s)\,\bbE[|Y^\theta_s|^2]\De s\,,
 \ees
 which combined with Gr\"onwall's lemma gives $\bbE[|Y^\theta_0|^2]\, e^{\int_0^s 2\lambda(\psi^\nu_t)\De t}\leq \bbE[|Y^\theta_s|^2]$, 
that integrated over $s\in[0,\delta T]$ reads as
\bes
\bbE[|Y^\theta_0|^2]\leq \biggl(\int_0^{\delta T}e^{\int_0^s 2\lambda(\psi^\nu_t)\De t}\De s\biggr)^{-1} \int_0^{\delta T}\bbE[|Y^\theta_s|^2]\De s=\frac{\Cpsi_\delta}{T}\int_0^{\delta T}\bbE[|Y^\theta_s|^2]\De s\,.
\ees
Next, notice that from Girsanov's theory (namely, the energy entropy identity~\eqref{eq:energia:entropia:forward}) we may recognize in the above right-hand side the relative entropy on the path space between the Schr\"odinger bridge from $\rho$ to $\mu$ and the Schr\"odinger bridge from $\rho$ to $\nu$, restricted on the time interval $[0,\delta T]$, that is  
\bes
\bbE[|Y^\theta_0|^2]\leq \frac{2\,\Cpsi_\delta}{T}\,\scrH(\cL(X^{\psi^\mu,\rho}_{[0,\delta T]})|\cL(X^{\psi^\nu,\rho}_{[0,\delta T]}))\,.
\ees
By recalling the time-reversal identities \eqref{eq:time:reversal:bridge} and by applying the disintegration property of relative entropies (cf.\ \cite[Lemma 1.6]{Marcel:notes} and \cite[Appendix A]{LeoSch}) and Girsanov's Theorem (w.r.t.\ the backward corrector process $Y^\eta_\cdot$) we deduce that
\bes
\begin{aligned}
\bbE[|Y^\theta_0|^2]\leq&\, \frac{2\,\Cpsi_\delta}{T}\,\scrH(\cL(X^{\varphi^\mu,\mu}_{[(1-\delta) T,T]})|\cL(X^{\varphi^\nu,\nu}_{[(1-\delta )T,T]}))\\
=&\,\frac{\Cpsi_\delta}{T}\biggl(2\,\scrH(\cL(X^{\varphi^\mu,\mu}_{(1-\delta) T})|\cL(X^{\varphi^\nu,\nu}_{(1-\delta )T}))+\int_{(1-\delta)T}^T\bbE[|Y^\eta_s|^2]\De s\biggr)\\
=&\,\frac{\Cpsi_\delta}{T}\biggl(2\,\scrH(\pi^\mu_{(1-\delta) T}|\pi^\nu_{(1-\delta )T})+\int_{(1-\delta)T}^T\bbE[|Y^\eta_s|^2]\De s\biggr)\,.
\end{aligned}
\ees
This proves our first claim. By recalling the identity~\eqref{eq:energia:entropia:running} we then have
\bes
\bbE[|Y^\theta_0|^2]\leq
\frac{\Cpsi_\delta}{T}\biggl(2\,\scrH(\pi^\mu_{(1-\delta) T}|\pi^\nu_{(1-\delta )T})+2\,\bbE_\mu[\scrH(\pi^{\mu}(\cdot|X)|\pi^\nu(\cdot|X))]\biggr)\,,
\ees
which can be bounded with \Cref{cor:bound_marg_entropy} and with \Cref{thm:entropic_stab}, yielding to our second claim.
\end{proof}

In theAppendix we specify  \Cref{thm:stability:mappe:general} to diverse settings and in Corollaries B.1 and B.2 there we prove the asymptotic bounds stated in \Cref{thm:intro} in the Introduction.

\subsection{Quantitative stability estimates of Hessian}

Let us consider once again the function $\theta_s\coloneqq \psi^\nu_s-\psi^\mu_s$ introduced in \Cref{sec:preliminari} and the forward process $(Y^\theta_s)_{s\in[0,T]}$ defined as $Y^\theta_s\coloneqq \nabla\theta_s(X^{\psi^\mu,\rho}_s)$, where  $(X^{\psi^\mu,\rho})_{s\in[0,T]}$ is the Schr\"odinger bridge~\eqref{eq:def:SB:dritto} (from $\rho$ to $\mu$), and recall that
\bes\begin{cases}
\partial_s \theta_s + \frac12\Delta \theta_s -\nabla \psi^\mu_s\cdot \nabla \theta_s -\frac12|\nabla\theta_s|^2=0\,,\\
\De X^{\psi^\mu,\rho}_s = -\nabla\psi^\mu_s(X^{\psi^\mu,\rho}_s)\De s+\De B_s\,,\quad X_0^{\psi^\mu,\rho}\sim\rho\,,\\
\De Y^\theta_s=\nabla^2\psi^\nu(X^{\psi^\mu,\rho}_s)Y^\theta_s\,\De s+\nabla^2\theta_s(X^{\psi^\mu,\rho}_s)\,\De B_s\,.
\end{cases}\ees
Next, let $Z^\theta_s\coloneqq \nabla^2\theta_s(X^{\psi^\mu,\rho}_s)$ and notice that
\bes
\begin{aligned}
\De Z^\theta_s =\Big[2\, \mathrm{sym} (Z^\theta_s \nabla^2 \psi_s^\mu(X^{\psi^\mu,\rho}_s)) + \nabla^3\psi_s^\nu(X^{\psi^\mu,\rho}_s)Y^\theta_s \Big]\,\De s + \nabla^3\theta_s(X^{\psi^\mu,\rho}_s)\De B_s\,,
\end{aligned}
\ees
where for any square matrix $M$ the symbol $\mathrm{sym}(M)\coloneqq (M+M^\intercal)/2$ denotes its symmetrized version and where for any $h\in\{\psi^\mu_s,\,\theta_s\}$ and $v\in\bbRD$ we have defined the product $\nabla^3h\,v$ as the matrix with entries $(\nabla^3h\,v)_{ij} :=\langle  \nabla(\partial_i\partial_j  h),\, v \rangle$.

Clearly, our goal when proving the Hessian stability result is getting a bound on $\bbE\HS{Z^\theta_0}$  since 
\bes
\|\nabla^2 \varphi^\nu - \nabla^2 \varphi^\mu\|_{\rmL^1(\rho)} = \|\nabla^2 \theta_0\|_{\rmL^1(\rho)}=\bbE\HS{\nabla^2\theta_0(X^{\psi^\mu,\rho}_0)}=\bbE\HS{Z^\theta_0}\,.
\ees
In view of that, let us firstly prove some lemmata where we are able to bound $\bbE\HS{Z^\theta_0}$ by means of the process $Y_\cdot$ and its norm.

\begin{lemma}\label{lemma:HS:Z0}
   Assume \Cref{ass:basic} and fix ${\tau_\ell}\in(0,T)$. Then we have 
    \bes\begin{aligned}
\bbE\HS{Z^\theta_0}\leq  \biggl[{\tau_\ell}^{-\nicefrac12}+2 \,{\tau_\ell}^{\nicefrac12}\,(\inf_{s\in[0,{\tau_\ell}]}\lambda(\psi^\nu_s))^-\biggr]\biggl(\int_0^{{\tau_\ell}}\bbE \HS{Z^\theta_s}^2\De s\biggr)^{\nicefrac12}\\
+\int_0^{{\tau_\ell}} \bbE \HS{Z^\theta_s}^2\De s
+\int_0^{{\tau_\ell}}\bbE\HS{\nabla^3\psi^\nu_s(X^{\psi^\mu,\rho}_s)\,Y^\theta_s}\De s\,,
    \end{aligned}
    \ees
    where the negative part of $a\in\rset$ is defined as $a^-\coloneqq \max\{-a,0\}$.
\end{lemma}

\begin{proof}
For notation's sake let $\Gamma^\theta_s=\nabla^3\theta_s(X^{\psi^\mu,\rho}_s)$ and note that by It\^o's formula we have
\bes\begin{aligned}
\De \HS{Z^\theta_s}^2= 2Z^\theta_s\,\De Z^\theta_s+\sum_{ijk}|\Gamma^{\theta,ijk}_s|^2\De s\,.
\end{aligned}
\ees
Hence, for any $\varepsilon\in(0,1)$ It\^o's formula for the function $r_\varepsilon(a)=\sqrt{a+\varepsilon}$ yields
\bes
\begin{aligned}
\De r_\varepsilon(\HS{Z^\theta_s}^2)=&\,\frac{Z^\theta_s\De Z^\theta_s}{r_\varepsilon(\HS{Z^\theta_s}^2)}+\frac{\sum_{ijk}|\Gamma^{\theta,ijk}_s|^2}{2r_\varepsilon(\HS{Z^\theta_s}^2)}\De s-\frac{\HS{Z^\theta_s\cdot\Gamma^\theta_s}^2}{2\,r_\varepsilon^3(\HS{Z^\theta_s}^2)}\De s \\
=&\,\biggl[- \frac{Z^\theta_s\cdot(Z^\theta_s)^2}{r_\varepsilon(\HS{Z^\theta_s}^2)}+2\frac{Z^\theta_s\cdot\mathrm{sym} (Z^\theta_s \nabla^2 \psi_s^\nu(X^{\psi^\mu,\rho}_s))}{r_\varepsilon(\HS{Z^\theta_s}^2)} +\frac{Z^\theta_s\cdot\nabla^3\psi_s^\nu(X^{\psi^\mu,\rho}_s)Y^\theta_s}{r_\varepsilon(\HS{Z^\theta_s}^2)}\biggr]\De s \\
&\quad+  \frac{Z^\theta_s\cdot \nabla^3\theta_s(X^{\psi^\mu,\rho}_s)}{r_\varepsilon(\HS{Z^\theta_s}^2)}\De B_s+\biggl[\frac{\sum_{ijk}|\Gamma^{\theta,ijk}_s|^2}{2r_\varepsilon(\HS{Z^\theta_s}^2)}-\frac{\HS{Z^\theta_s\cdot\Gamma^\theta_s}}{2\,r_\varepsilon^3(\HS{Z^\theta_s}^2)}\biggr] \,\De s \,.
\end{aligned}
\ees
Next observe that from Cauchy--Schwarz inequality the last term above is almost surely non-negative since
\bes
\frac{\HS{Z^\theta_s\cdot\Gamma^\theta_s}^2}{2\,r_\varepsilon^3(\HS{Z^\theta_s}^2)}=\frac{\sum_k |\sum_{ij}Z^{\theta,ij}_s\Gamma^{\theta,ijk}_s|^2}{2\,r_\varepsilon^3(\HS{Z^\theta_s}^2)}\leq \frac{\HS{Z^\theta_s}^2\, \sum_{ijk}|\Gamma^{\theta,ijk}_s|^2}{2\,r_\varepsilon^3(\HS{Z^\theta_s}^2)}\leq \frac{ \sum_{ijk}|\Gamma^{\theta,ijk}_s|^2}{2\,r_\varepsilon(\HS{Z^\theta_s}^2)}\,.
\ees
Let us now provide a lower bound for each of the terms. For the first one, we use first Cauchy--Schwarz inequality and the sub-multiplicative property of the HS norm to obtain
\bes
\frac{Z^\theta_s\cdot(Z^\theta_s)^2}{r_\varepsilon(\HS{Z^\theta_s}^2)}=\frac{\sum_{ij}Z^{\theta,ij}_s(Z^{\theta}_s\cdot Z^{\theta}_s)^{ij}}{r_\varepsilon(\HS{Z^\theta_s}^2)}\leq \frac{\HS{Z^\theta_s}\,\HS{(Z^\theta_s)^2}}{r_\varepsilon(\HS{Z^\theta_s}^2)}\leq \HS{Z^\theta_s}^2\,.
\ees
For the second one, we first use the fact that $Z^\theta$ and $\nabla^2\psi$ are symmetric, the permutation identities and the monotonicity of the trace in order to rewrite it as
\bes
\begin{aligned}
2&\frac{Z^\theta_s\cdot\mathrm{sym} (Z^\theta_s \nabla^2 \psi_s^\nu(X^{\psi^\mu,\rho}_s))}{r_\varepsilon(\HS{Z^\theta_s}^2)} 
=\frac{\Tr(Z^\theta_s \cdot Z^\theta_s \nabla^2 \psi_s^\nu(X^{\psi^\mu,\rho}_s))+\Tr(Z^\theta_s\cdot  \nabla^2 \psi_s^\nu(X^{\psi^\mu,\rho}_s) Z^\theta_s)}{r_\varepsilon(\HS{Z^\theta_s}^2)}\\
&\,=\frac2{r_\varepsilon(\HS{Z^\theta_s}^2)}\,\Tr(Z^\theta_s\cdot  \nabla^2 \psi_s^\nu(X^{\psi^\mu,\rho}_s) Z^\theta_s)\geq \frac{2\,\lambda(\psi^\nu_s)\,\HS{Z^\theta_s}^2}{r_\varepsilon(\HS{Z^\theta_s}^2)}\geq 2\,\lambda(\psi^\nu_s)\,\HS{Z^\theta_s}-2\,\lambda(\psi^\nu_s)\,\varepsilon\,.
\end{aligned}
\ees
For the third term we use again Cauchy--Schwarz inequality to obtain
\bes\begin{aligned}
&\frac{Z^\theta_s\cdot\nabla^3\psi_s^\nu(X^{\psi^\mu,\rho}_s)Y^\theta_s}{r_\varepsilon(\HS{Z^\theta_s}^2)}=\frac{\sum_{ijk}Z^{\theta,ij}_s\partial_{ijk}\psi^\nu_s(X^{\psi^\mu,\rho}_s)Y^{\theta,k}_s}{r_\varepsilon(\HS{Z^\theta_s}^2)}\geq -\HS{\nabla^3\psi^\nu_s(X^{\psi^\mu}_s)\,Y^\theta_s}\,.
\end{aligned}\ees
We have thus shown that for any $\varepsilon\in(0,1)$ almost surely it holds 
\bes
\begin{aligned}
\De r_\varepsilon(\HS{Z^\theta_s}^2)\geq \biggl(-\HS{Z^\theta_s}^2+2\,\lambda(\psi^\nu_s)\,\HS{Z^\theta_s}-\HS{\nabla^3\psi^\nu_s(X^{\psi^\mu}_s)\,Y^\theta_s}-2\,\lambda(\psi^\nu_s)\,\varepsilon\biggr)\De s\\
 +\frac{Z^\theta_s\cdot \nabla^3\theta_s(X^{\psi^\mu,\rho}_s)}{r_\varepsilon(\HS{Z^\theta_s}^2)}\De B_s\,.
\end{aligned}
\ees
Taking expectation and integrating for $s\in[0,t]$ we get
\bes\begin{aligned}
\bbE\HS{Z^\theta_0}\leq \bbE[ r_\varepsilon(\HS{Z^\theta_0}^2)]\leq\bbE [r_\varepsilon(\HS{Z^\theta_t}^2)]+2\varepsilon\int_0^t\lambda(\psi^\nu_s)\De s+\int_0^t\bbE\HS{Z^\theta_s}^2\De s\\
-2\int_0^t\lambda(\psi^\nu_s)\bbE\HS{Z^\theta_s}\De s+\int_0^t\bbE\HS{\nabla^3\psi^\nu_t(X^{\psi^\mu,\rho}_s)\,Y^\theta_s}\De s \,,
\end{aligned}
\ees
which combined with the Dominated Convergence Theorem, for $\varepsilon\downarrow 0$, implies
\bes\begin{aligned}
\bbE\HS{Z^\theta_0} \leq\bbE \HS{Z^\theta_t}+\int_0^t\bbE\HS{Z^\theta_s}^2\De s-2\int_0^t\lambda(\psi^\nu_s)\bbE\HS{Z^\theta_s}\De s
+\int_0^t\bbE\HS{\nabla^3\psi^\nu_t(X^{\psi^\mu,\rho}_s)\,Y^\theta_s}\De s \,.
\end{aligned}
\ees
Finally, by integrating over $t\in[0,{\tau_\ell}]$ we conclude that
\bes
\begin{aligned}
{\tau_\ell}\, \bbE\HS{Z^\theta_0}\leq\, \int_0^{{\tau_\ell}}\bbE \HS{Z^\theta_t}\De t+\int_0^{{\tau_\ell}} \int_0^t\bbE\HS{Z^\theta_s}^2\De s\De t-2 \int_0^{{\tau_\ell}}\int_0^t\lambda(\psi^\nu_s)\bbE\HS{Z^\theta_s}\De s\De t\\
+\int_0^{{\tau_\ell}}\int_0^t\bbE\HS{\nabla^3\psi^\nu_t(X^{\psi^\mu,\rho}_s)\,Y^\theta_s}\De s\De t\\
\leq\, (1+2 \,\tau_\ell\,(\inf_{s\in[0,{\tau_\ell}]}\lambda(\psi^\nu_s))^-\,) \int_0^{{\tau_\ell}}\bbE \HS{Z^\theta_s}\De s
+{\tau_\ell}\int_0^{{\tau_\ell}} \bbE \HS{Z^\theta_s}^2\De s\\
+{\tau_\ell}\int_0^{{\tau_\ell}}\bbE\HS{\nabla^3\psi^\nu_t(X^{\psi^\mu,\rho}_s)\,Y^\theta_s}\De s\,,
\end{aligned}
\ees
Applying Jensen's inequality concludes our proof.

\end{proof}

Let us now fix $\delta'<\delta\in[0,T]$ arbitrary and consider the constant
\be\label{eq:def:C:delta:primo:delta}
\Cpsi_{\delta',\delta} \coloneqq T\biggl(\int_{\delta'T}^{\delta T}e^{\int_{\delta' T}^{s}2\lambda(\psi^\nu_t)\De t}\De s\biggr)^{-1} \,,
\ee
which generalizes the constant $\Cpsi_\delta$ considered in \Cref{thm:stability:mappe:general}. By repeating the same argument employed in \Cref{thm:entropic_stab} when proving the upper bound \eqref{eq:bound:Y:eta:0} for $\bbE[|Y^\theta_T|^2]=\bbE[|Y^\eta_0|^2]$,  we can prove the following generalization.

\begin{lemma}\label{lem:Y_delta'_bound}
Assume \Cref{ass:basic} and \Cref{ass:mu}. For any fixed $\delta'< \delta\in[0,1]$ we have
\bes
\bbE[|Y^\theta_{\delta' T}|^2]\leq \frac{C^{\psi^\nu}_{\delta',\delta}}{T}\,\biggl(\frac{3\,\Lambda(\varphi^\nu_0)}{T}+\frac{\delta}{1-\delta}\,\frac1T+\frac{2\,\sqrt{\Lam{\varphi^\nu_0}\,\Cphi}}{T}\biggr)\,\bfW_2^2(\mu,\nu)\,.
\ees
\end{lemma}
\begin{proof}
By reasoning as in the proof of \Cref{thm:entropic_stab}, 
from It\^o's formula  and \eqref{eq:sde:Ytheta}, we have
 \bes
\De \bbE[|Y^\theta_s|^2]\geq 2\,\bbE[Y^\theta_s\cdot \nabla^2\psi^\nu_s(X^{\psi^\mu,\rho}_s)Y^\theta_s ]\De s \geq 2\lambda(\psi^\nu_s)\,\bbE[|Y^\theta_s|^2]\De s\,, \quad \forall s \leq \delta T \,,
\ees
which combined with Gr\"onwall's lemma gives $\bbE[|Y^\theta_{\delta' T}|^2]\, e^{\int_{\delta'T}^s 2\lambda(\psi^\nu_t)\De t}\leq \bbE[|Y^\theta_s|^2]$.

Integrating this inequality over $s\in[\delta' T,\delta T]$ gives
\bes
\bbE[|Y^\theta_{\delta'T}|^2]\leq \biggl(\int_{\delta' T}^{\delta T}e^{\int_{\delta'T}^s 2\lambda(\psi^\nu_t)\De t}\De s\biggr)^{-1} \int_{\delta' T}^{\delta T}\bbE[|Y^\theta_s|^2]\De s\leq\frac{\Cpsi_{\delta',\delta}}{T}\int_0^{\delta T}\bbE[|Y^\theta_s|^2]\De s\,.
\ees
Given the above, the thesis follows from \Cref{prop:integral:Y}. 
\end{proof}

Next, we give a bound for the time integral of $\bbE\HS{Z^\theta_s}$ appearing in \Cref{lemma:HS:Z0}.

\begin{lemma}\label{lemma:int:HS:Z}
Assume \Cref{ass:basic} and \Cref{ass:mu}. For any fixed $\delta'\leq \delta\in[0,T]$ we have
\bes
\int_{0}^{\delta'T}\bbE\HS{Z^{\theta}_s}^2\De s \leq \frac{\Kappadel}{T^2}\,\bfW_2^2(\mu,\nu) \,,
\ees
where the constant is defined as
\be\label{eq:def:Kappadel}\begin{aligned}
\Kappadel\coloneqq&\,2 C^{\psi^\nu}_{\delta',\delta}\,\biggl(3\,\Lambda(\varphi^\nu_0)+\frac{\delta}{1-\delta}+2\,\sqrt{\Lam{\varphi^\nu_0}\,\Cphi}\biggr)\\
+&\,4\,T\,(\inf_{s\in[0,\delta'T]}\lambda(\psi^\nu_s))^-\biggl(3\,\Lambda(\varphi^\nu_0)+\frac{\delta'}{1-\delta'}+2\,\sqrt{\Lam{\varphi^\nu_0}\,\Cphi}\biggr)\,.
\end{aligned}
\ee
\end{lemma}
\begin{proof}
From It\^o's formula and \eqref{eq:sde:Ytheta}, by taking expectation we have
\bes
\frac{\De}{\De s} \bbE|Y^\theta_s|^2\geq 2\lambda(\psi^\nu_s)\,\bbE|Y^\theta_s|^2+\frac12\bbE\HS{Z^\theta_s}^2\,,
\ees
which integrated over $s\in[0,\delta' T]$ leads to
\bes
\int_0^{\delta' T}\bbE\HS{Z^\theta_s}^2\De s\leq 2\,\bbE|Y^\theta_{\delta' T}|^2+4\,(\inf_{s\in[0,\delta'T]}\lambda(\psi^\nu_s))^-\int_0^{\delta'T}\bbE|Y^\theta_s|^2\De s\,.
\ees
Then our thesis can be obtained by bounding the first term with \Cref{lem:Y_delta'_bound} and the second term as already done in \Cref{prop:integral:Y}.
\end{proof}

Our last ingredient is an upper bound for the time integral of the third derivative term.

\begin{proposition}\label{lemma:nabla3}
Assume \Cref{ass:basic}. Fix ${\tau_u}\in(0,T]$. Then for all $t\in (0,{\tau_u}]$,
\bes
\|\nabla^3\psi^\nu_t(x)[v]\|_{\rm HS} \leq |v|\,\biggl(\frac1{\tau_u-t}+2\gamma_{\tau_u}\biggr)\,\frac{2\gamma_{\tau_u}}{\sqrt{2\pi}}\,\int_t^{\tau_u}\SymIpsi{t}{s}{\psi^\nu}^{-\nicefrac12}\De s\,,
\ees
where 
\[
\gamma_{\tau_u}\coloneqq \sup_{s\in[0,{\tau_u}]}\sup_{x\in\bbRD}\HS{\nabla^2\psi^\nu_s} \quad\textrm{and}\quad
\SymIpsi{t}{s}{\psi^\nu}\coloneqq \int_t^s\exp\biggl(\int_t^u 2\lambda(\psi^\nu_l)\De l\biggr)\De u \,.
\]
\end{proposition}

\begin{proof}
Fix $x,\widehat{x}\in \bbR^d$. Our aim is controlling $
\|\nabla^2 \psi^\nu_t (x) - \nabla^2 \psi^\nu_t (\widehat{x})\|_{\rm HS}$ with $|x-\widehat{x}|$.
In view of this, let us consider the processes $X^{t,x}_\cdot$ and $X^{t,\widehat{x}}_\cdot$ satisfying for $s\in[t,{\tau_u}]$
\bes
\begin{cases}
  \De X^{t,x}_s = - \nabla \psi^{\nu}_s(X^{t,x}_s)\, \De s + \De B_s \,, \\  
  \De X^{t,\widehat{x}}_s = - \nabla \psi^{\nu}_s(X^{t,\widehat{x}}_s)\, \De s + \De \hat{B}_s \,, \quad\forall\eqsp t\in[0,{\tau}_{\mathrm{st}})\eqsp\text{ and }X^{t,\widehat{x}}_s=X^{t,x}_s\quad\forall \eqsp s\geq {\tau}_{\mathrm{st}}\\
 X^{t,x}_t = x\text{ and }X^{t,\widehat{x}}_t = \widehat{x}\,,
\end{cases}
\ees
where ${\tau}_{\mathrm{st}}\coloneqq \inf\{s\geq t:X^{t,x}_s=X^{t,\widehat{x}}_s\}\wedge {\tau_u}$, 
and $(\hat{B}_s)_{s\geq t}$ is defined as
\begin{equation*}
\De \hat{B}_s\coloneqq (\operatorname{Id}-2\,e_s\,e_s^{{\sf  T}}\,\mathbf{1}_{\{s< {\tau}_{\mathrm{st}}\}})\, \De B_s\qquad\text{where}\quad e_s\coloneqq \begin{cases}\frac{X^{t,x}_s-X^{t,\widehat{x}}_s}{|X^{t,x}_s-X^{t,\widehat{x}}_s|}\quad&\text{ when } r_t>0\,,\\
   u\quad&\text{ when }r_t=0\,,\end{cases}
\end{equation*}
where $r_t\coloneqq |X^{t,x}_s-X^{t,\widehat{x}}_s|$ and $u\in\bbRD$ is a fixed (arbitrary) unit-vector. By L\'evy's characterization, $(\hat{B}_t)_{t\geq 0}$ is a $d$-dimensional Brownian motion, therefore $X^{t,x}_\cdot$ and $X^{t,\widehat{x}}_\cdot$ are two Schr\"odinger bridge processes (from $\rho$ to $\nu$) started respectively in $x$ and $\widehat{x}$, coupled via the coupling by reflection.

Let us also consider the processes $Z_s = \nabla^2\psi^\nu_s(X_s^{t,x})$ and $\widehat{Z}_s = \nabla^2\psi^\nu_s(X_s^{t,\widehat{x}})$. 
Since
\bes
\partial_s\nabla^2\psi^\nu_s+\frac12\Delta\nabla^2\psi^\nu_s-\nabla^3\psi^\nu_s\nabla\psi^\nu_s-(\nabla^2\psi^\nu_s)^2=0\,,
\ees
 by means of It\^o's formula we have 
\bes
\De Z_s = Z_s^2 \De s + \nabla^3\psi^\nu(X^{t,x}_s)\De B_s\,,\qquad \De \widehat{Z}_s = \widehat{Z}_s^2 \De t + \nabla^3\psi^\nu(X^{t,\widehat{x}}_s)\De \hat{B}_s\,.
\ees
Therefore, if we set $\De M_s\coloneqq \nabla^3\psi^\nu(X^{t,x}_s)\De B_s-\nabla^3\psi^\nu(X^{t,\widehat{x}}_s)\De \hat{B}_s$, from It\^o's formula we first deduce that
\bes
\De \HS{Z_s-\widehat{Z}_s}^2=2(Z_s-\widehat{Z}_s)\cdot(Z_s^2-\widehat{Z}_s^2)\De s+\sum_{i,j}\De [M^{ij}_\cdot]_s+2(Z_s-\widehat{Z}_s)\cdot\De M_s
\ees
where the $A\cdot B$ corresponds to the Hilbert--Schmidt scalar product between the two matrices $A,\,B$ that is the scalar $\sum_{i,j}A^{ij}B^{ij}$. 

From another application of It\^o's formula (as we already did in the proof of \Cref{lemma:HS:Z0}, by firstly applying it to $r_\varepsilon(a)\coloneqq \sqrt{a+\varepsilon}$ and then let $\varepsilon\downarrow 0$ ) 
we then have
\bes
\begin{aligned}
\HS{Z_s-\widehat{Z}_s}=\frac{(Z_s-\widehat{Z}_s)\cdot(Z_s^2-\widehat{Z}_s^2)}{\HS{Z_s-\widehat{Z}_s}}\De s+\frac{(Z_s-\widehat{Z}_s)\cdot\De M_s}{\HS{Z_s-\widehat{Z}_s}}
+\frac{\sum_{i,j}\De [M^{ij}_\cdot]_s}{2\,\HS{Z_s-\widehat{Z}_s}}-\frac{(Z_s-\widehat{Z}_s)^2\cdot\De[M_\cdot]_s}{2\HS{Z_s-\widehat{Z}_s}^3}\,.
\end{aligned}
\ees
Since $Z_s$ and $\widehat{Z}_s$ are symmetric matrices we have $Z_s\cdot(\widehat{Z}_s Z_s)=Z_s\cdot(Z_s\widehat{Z}_s)$ and, by recalling $\HS{\nabla^2\psi_s^\nu}\leq \gamma_{\tau_u}$ for any $s\in(0,{\tau_u}]$, we then have from Cauchy--Schwarz inequality that 
\bes
\begin{aligned}
(Z_s-\widehat{Z}_s)\cdot(Z_s^2-\widehat{Z}_s^2)=(Z_s-\widehat{Z}_s)^2\cdot (Z_s+\widehat{Z}_s)\geq -\HS{Z_s-\widehat{Z}_s}\,\HS{Z_s+\widehat{Z}_s}\geq -2\gamma_{\tau_u}\HS{Z_s-\widehat{Z}_s}\,.
\end{aligned}
\ees
Moreover the two quadratic covariation terms cancel out since 
\bes
(Z_s-\widehat{Z}_s)^2\cdot\De[M_\cdot]_s=\sum_{i,j}(Z^{ij}_s-\widehat{Z}^{ij}_s)^2\De [M^{ij}_\cdot]_s
\leq \HS{Z_s-\widehat{Z}_s}^2\sum_{i,j}\De [M^{ij}_\cdot]_s\,. 
\ees
Putting these two remarks together yields
\bes
\begin{aligned}
\HS{Z_s-\widehat{Z}_s}
\geq-2\gamma_{\tau_u} \HS{Z_s-\widehat{Z}_s}\De s +\frac{Z_s-\widehat{Z}_s}{\HS{Z_s-\widehat{Z}_s}}\De M_s\,,
\end{aligned}
\ees
 which implies
\bes
\frac{\De}{\De s}\bbE\HS{Z_s-\widehat{Z}_s}\geq -2\gamma_{\tau_u}\bbE\HS{Z_s-\widehat{Z}_s}\De s\,,
\ees
and hence that
\be\label{eq:Zt:tau+integral}
 \HS{\nabla^2 \psi^\nu_t (x) -\nabla^2 \psi^\nu_t (\widehat{x})}  = \bbE\HS{Z_t - \widehat{Z}_t}  \leq \bbE\HS{Z_{\tau_u}-\widehat{Z}_{\tau_u}}+2\gamma_{\tau_u}\,\int_t^{\tau_u}\bbE\HS{Z_s-\widehat{Z}_s}\De s\,.
\ee
Next, notice that for any $s\in[t,\tau_u]$ we can write
\be\label{eq:Zs:gamma:prob}
\HS{Z_s - \widehat{Z}_s}
    = \bbE\biggl[\HS{Z_s - \widehat{Z}_s}\,\mathbf{1}_{\{X^{t,x}_s\neq X^{t,\widehat{x}}_s\}}\biggr]
    \leq 2\gamma_{\tau_u} \bbP(X^{t,x}_s \neq X^{t,\widehat{x}}_s)\,.
\ee
Henceforth, the rest of the proof deals with  estimating $\bbP(X^{t,x}_s \neq X^{t,\widehat{x}}_s)$ for any $s\in[t,\tau_u]$. To do so we look at the one-dimensional process $r_s=|X_s^{t,x} - X_s^{t,\widehat{x}}|$,  so that $\bbP(X^{t,x}_s \neq X^{t,\widehat{x}}_s) = \bbP(r_s > 0)$. From It\^o's formula we get
\bes
\De r_s^2=(-2(X_{s}^{t,x} - X_{s}^{t,\widehat{x}})(\nabla \psi^\nu_s (X_{s}^{t,x}) - \nabla \psi^\nu_s (X_{s}^{t,\widehat{x}}))+4)\De s+4\,r_s\,\De W_s \,,
\ees
where $\De W_s= e_s^\intercal\De B_s$ is a one-dimensional Brownian motion. Therefore another application of It\^o's formula yields 
 \bes
\De r_s=-e_s(\nabla \psi^\nu_s (X_{s}^{t,x}) - \nabla \psi^\nu_s (X_{s}^{t,\widehat{x}}))\De s+2\,\De W_s\leq 
-\lambda(\psi^\nu_s)\,r_s\,\De s+2\,\De W_s\,.
 \ees
Hence the process $r_\cdot$ is dominated from above by the process $\widetilde{r}_\cdot$ which solves for $s\in [t,{\tau_u}]$
\bes
\De \widetilde{r}_s = -\lambda(\psi^\nu_s)\, \widetilde{r}_s\,\De s + 2\,\De W_s\,,\qquad \widetilde{r}_t = |x-\widehat{x}|\,.
\ees
Moreover, notice that the above SDE implies that the process defined for any $s\in [t,{\tau_u}]$ as  $N_s \coloneqq e^{\int_t^s\lambda(\psi^\nu_{u})\De u}\,\widetilde{r}_{s}$ is a martingale, more precisely
\bes
\De N_s=2\exp\biggl(\int_t^s \lambda(\psi^\nu_u)\De u\biggr)\De W_s,\quad \text{ with }N_t=|x-\widehat{x}|\,.
\ees
Therefore from the Martingale Representation Theorem we have $N_s=N_t+B_{[N]_s}$ where $B_\cdot$ is a Brownian motion and 
\bes
[N]_s= 4\int_t^s\exp\biggl(\int_t^u 2\lambda(\psi^\nu_l)\De l\biggr)\De u\,.
\ees
This information can then be employed in bounding $\bbP(X^{t,x}_{\tau_u} \neq X^{t,\widehat{x}}_{\tau_u}) = \bbP(r_{\tau_u} > 0)$ since from the Reflection Principle we may deduce
\bes
\begin{aligned}
\bbP(X_s^{t,x} \neq X_s^{t,\widehat{x}}) =&\, \bbP(r_s > 0 )=\bbP\Big(\inf_{u \in [t,s]}r_u > 0 \Big)\leq \bbP\Big(\inf_{u \in [t,s]}\widetilde{r}_u > 0 \Big)\leq \bbP\Big(\inf_{u \in [t,s]}N_u > 0 \Big)\\ 
=&\,\bbP\Big(\inf_{u\in[t,s]} B_{[N]_u}>-|x-\widehat{x}|\Big)=\bbP\Big(\sup_{u\in[t,[N]_s]} B_{u}\leq |x-\widehat{x}|\Big)=\bbP\Big( | B_{[N]_s}| \leq  |x-\widehat{x}| \Big)\\
\leq&\, \sqrt{\frac{2}{\pi}}\,|x-\widehat{x}| \,[N]_s^{-\nicefrac12}= \frac{|x-\widehat{x}|}{\sqrt{2\pi}}\biggl(\int_t^s\exp\biggl(\int_t^u 2\lambda(\psi^\nu_l)\De l\biggr)\De u\biggr)^{-\nicefrac12}.
\end{aligned}
\ees
By combining this last estimate with \eqref{eq:Zs:gamma:prob} in \eqref{eq:Zt:tau+integral} gives
\bes
\begin{aligned}
\HS{\nabla^2 \psi^\nu_t (x) -\nabla^2 \psi^\nu_t (\widehat{x})}\leq&\, 2\gamma_{\tau_u}\frac{|x-\widehat{x}|}{\sqrt{2\pi}}\biggl(\SymIpsi{t}{\tau_u}{\psi^\nu}^{-\nicefrac12}+2\gamma_{\tau_u}\int_t^{\tau_u}\SymIpsi{t}{s}{\psi^\nu}^{-\nicefrac12}\De s\biggr)\\
\leq&\,2\gamma_{\tau_u}\,\biggl(\frac1{\tau_u-t}+2\gamma_{\tau_u}\biggr)\,\frac{|x-\widehat{x}|}{\sqrt{2\pi}}\,\int_t^{\tau_u}\SymIpsi{t}{s}{\psi^\nu}^{-\nicefrac12}\De s\,,
\end{aligned}\ees
and hence the conclusion.
\end{proof}

We are now ready to prove the general quantitative stability result for the Hessians. 

\begin{theorem}[Stability of Hessians (with explicit costants)]\label{thm:stability_hessians:general}
    Assume \Cref{ass:basic} and \Cref{ass:mu}. For any $\delta'< \delta\in[0,1]$ we have
    \bes 
    \| \nabla^2 \varphi^\mu -  \nabla^2 \varphi^\nu\|_{\rmL^1(\rho)} \leq 
   A\,\bfW_2(\mu,\nu)+\frac{\Kappadel}{T^2}\,\bfW_2^2(\mu,\nu)\,, 
    \ees
    with $A$ defined at \eqref{eq:def:A} and $\Kappadel$ defined at \eqref{eq:def:Kappadel}.
    \end{theorem}

\begin{proof}
Fix $\delta'< \delta\in[0,1]$, let $\tau_u=\delta T$, $\tau_\ell=\delta' T$, fix the positive constants 
\bes
\gamma_{\tau_u}\coloneqq \sup_{s\in[0,\tau_u]}\sup_{x\in\bbRD}\HS{\nabla^2\psi^\nu_s}\,,\quad\text{ and }\bar\lambda_{\tau_\ell}\coloneqq (\inf_{s\in[0,{\tau_\ell}]}\lambda(\psi^\nu_s))^-\,.
\ees
From \Cref{lemma:nabla3}, \Cref{prop:integral:Y} and H\"older's inequality we see that 
\bes\begin{aligned}
\int_0^{\tau_\ell}\bbE&\HS{\nabla^3\psi^\nu_s(X^{\psi^\mu,\rho}_s)\,Y^\theta_s}\De s\\
\leq&\, \biggl(\frac1{\tau_u-\tau_\ell}+2\gamma_{\tau_u}\biggr)\,\frac{2\gamma_{\tau_u}}{\sqrt{2\pi}}\,\biggl(\int_0^{\tau_\ell}\bbE[|Y^\theta_s|^2]\De s\biggr)^{\nicefrac12}\,\tau_\ell\,\sup_{s\in[0,\tau_\ell]}\int_s^{\tau_u}\SymIpsi{s}{u}{\psi^\nu}^{-\nicefrac12}\De u\\
\leq &\,\frac{\bfW_2(\mu,\nu)}{\sqrt{T}}\biggl(3\,\Lambda(\varphi^\nu_0)+\frac{\delta'}{1-\delta'}+2\,\sqrt{\Lam{\varphi^\nu_0}\,\Cphi}\biggr)^{\nicefrac12}\biggl(\frac1{\tau_u-\tau_\ell}+2\gamma_{\tau_u}\biggr)\,\cdot\\
&\qquad\qquad\qquad\qquad\qquad\qquad\qquad\qquad\qquad\cdot\frac{2\gamma_{\tau_u}\,\tau_\ell}{\sqrt{2\pi}}\,\sup_{s\in[0,\tau_\ell]}\int_s^{\tau_u}\SymIpsi{s}{u}{\psi^\nu}^{-\nicefrac12}\De u\,.
\end{aligned}
    \ees
By combining \Cref{lemma:HS:Z0} with the above estimate and with \Cref{lemma:int:HS:Z} we finally get our thesis
with 
\be\label{eq:def:A}
\begin{aligned}
A\coloneqq \biggl[{\tau_\ell}^{-\nicefrac12}+2 \,{\tau_\ell}^{\nicefrac12}\,\bar\lambda_{\tau_\ell}\biggr]\,\frac{\sqrt{\Kappadel}}{T}
+\frac1{\sqrt{T}}\biggl(3\,\Lambda(\varphi^\nu_0)+\frac{\delta'}{1-\delta'}+2\,\sqrt{\Lam{\varphi^\nu_0}\,\Cphi}\biggr)^{\nicefrac12}\cdot\\
\cdot\biggl(\frac1{\tau_u-\tau_\ell}+2\gamma_{\tau_u}\biggr)\,\frac{2\gamma_{\tau_u}\,\tau_\ell}{\sqrt{2\pi}}\,\sup_{s\in[0,\tau_\ell]}\int_s^{\tau_u}\SymIpsi{s}{u}{\psi^\nu}^{-\nicefrac12}\De u\,.
\end{aligned}\ee
\end{proof}

Our estimates depend on the two free parameters $\delta'<\delta\in[0,1]$. A priori, one could simply optimize over their choice; however, this optimization heavily depends on the semiconcavity parameters. 
In the Appendix, in \Cref{cor:hess:compact} and  \Cref{cor:hess:log:conc}, we specify \Cref{thm:stability_hessians:general} to diverse settings by fixing appropriate choices for $\delta$ and $\delta'$, proving the specialized bounds stated in \Cref{thm:intro}.

\medskip

We conclude with the proof of the convergence of gradient and Hessian of Sinkhorn iterates. This will be a straightforward application of our quantitative stability estimates.

\begin{proof}[Proof of \Cref{thm:sinkhorn:hessian}]
    Under our assumptions, Talagrand inequality~\eqref{eq:TI} and the data processing inequality for relative entropy combined with \cite[Theorem 1.2]{chiarini2024semiconcavity}  guarantee that 
    \be\label{eq:cor:from:sink}
    \bfW_2^2(\mu,\mu^{n+1,n})\leq2\tau\,\scrH(\mu|\mu^{n+1,n})\leq 2\tau\,\scrH(\pi^{\mu}|\pi^{n+1,n}) \leq 2\bigg(1-\frac{  T }{T+\tau \Lambda}\bigg)^{(n-N+1)} \tau\scrH(\pi^{\mu}|\pi^{0,0})\,.
    \ee
    In particular $\scrH(\mu|\mu^{n+1,n})<+\infty$, and hence the validity of \Cref{ass:mu} for the marginal $\mu^{n+1,n}$.  This allows us to apply \Cref{thm:stability:mappe:general,thm:stability_hessians:general} (with the pair $\nu,\mu$ there, replaced here as $\mu,\mu^{n+1,n}$) and deduce
    \be\label{eq:cor:from:stab:mappe}
    \begin{aligned}
&\|\nabla\varphi^{n+1}-\nabla\varphi^\mu \|^2_{\rmL^2(\rho)} \leq \frac{\Crhomu^\delta}{T^2}\,\bfW_2^2(\mu^{n+1,n},\mu)\,,\\
   & \|    \nabla^2 \varphi^{{n+1}}-\nabla^2 \varphi^\mu\|_{\rmL^1(\rho)} \leq 
   A\,\bfW_2(\mu^{n+1,n},\mu)+\frac{\Kappadelmu}{T^2}\,\bfW_2^2(\mu^{n+1,n},\mu)\,, 
    \end{aligned}\ee
    with $\Crhomu$, $A$ and $\Kappadelmu$ defined as in the stability results, this time depending solely on $T$, and on the marginals $\rho$ and $\mu$.
    Putting together~\eqref{eq:cor:from:sink} and~\eqref{eq:cor:from:stab:mappe} leads to our general bounds.
Finally, the specific value of the uniform semiconcavity parameter $\Lambda$ and the asymptotics of the constants when $\rho$ and $\mu$ are compactly supported or log-concave can be obtained from the explicit computations performed in the Appendix (considering the pair of marginals $(\rho,\mu)$ as fixed and with $\mu^n$ seen as perturbation of $\mu$).
\end{proof}

\appendix

\section{Explicit computations for $\Cphi$ and $\Lam{\varphi^\nu_0}$}\label{sec:app:phi}
In this section we specify the constants appearing in the entropic stability bound of Theorem 2.1 to various settings. Before actually doing it, let us preliminary recall the well-known identities \cite{chewi2022entropic,fathi2019proof,conforti2023Sinkhorn,conforti2024weak} 
\be\label{Hess:Cov:psi:s}
\nabla^2\psi^\nu_s(y)=(T-s)^{-1}-(T-s)^{-2}\mathrm{Cov}(X^{\psi^\nu,\rho}_T|X^{\psi^\nu,\rho}_s=y)\quad\forall\,s\in[0,T)\,,
\ee
where $(X^{\psi^\nu,\rho}_s)_{s\in[0,T]}$ is the forward Schr\"odinger process (from $\rho$ to $\nu$) which we recall here to be defined
as
\bes
\De X^{\psi^\nu,\rho}_s = -\nabla\psi^\nu_s(X^{\psi^\nu,\rho}_s)\De s+\De B_s\,,\quad X_0^{\psi^\nu,\rho}\sim\rho\,,
\ees
 whereas $\mathrm{Cov}(X^{\psi^\nu,\rho}_0|X^{\psi^\nu,\rho}_s=y)$ is the covariance of the law of this process at initial time conditioned on being in $y$ at time $s$. This can be easily seen by recalling that
\bes
\psi^\nu_s(y)=-\log P_{T-s} e^{-\psi^\nu_T}(y) = -\log\int \exp\left(-\psi^\nu_T(x)-\frac{|x-y|^2}{2(T-s)}\right)\De x+\frac{d}{2}\log(2\pi (T-s))\,,
\ees
and computing the Hessian as done in \cite[Proposition 17]{conforti2023Sinkhorn} for the case $s=T$. 

Similarly, for $\varphi^\nu_s$, for any $s\in[0,T)$ we have 
\be\label{Hess:Cov:phi:s}
\nabla^2\varphi^\nu_s(y)=(T-s)^{-1}-(T-s)^{-2}\mathrm{Cov}(X^{\varphi^\nu,\nu}_T|X^{\varphi^\nu,\nu}_s=y)\,,
\ee
where $(X^{\varphi^\nu,\nu}_s)_{s\in[0,T]}$ is the backward Schr\"odinger process defined as
\bes
\De X^{\varphi^\nu,\nu}_s = -\nabla\varphi^\nu_s(X^{\varphi^\nu,\nu}_s)\De s+\De B_s\,,\quad X^{\varphi^\nu,\nu}_0\sim\nu\,,
\ees
 and $\mathrm{Cov}(X^{\varphi^\nu,\nu}_0|X^{\varphi^\nu,\nu}_s=y)$ is the covariance of the law of this process at initial time conditioned on being in $y$ at time $s$.

Furthermore, let us recall here the following convexity backpropagation result  along Hamilton-Jacobi-Bellman equations (see for instance \cite[Lemma 3.1]{conforti2024weak}) 
\begin{lemma}\label{lemma:backpropagation:convexity}
    Assume that $\nabla^2h\geq \alpha$ for some $\alpha>-T^{-1}$ uniformly. Then if $(h_s)_{s\in[0,T]}$ denotes the solution of 
    \bes
\begin{cases}
    \partial_s u_s+\frac12\Delta u_s-\frac12|\nabla u_s|^2=0\\
    u_s=h
\end{cases}
    \ees
then for any $s\in[0,T]$ we have $\nabla^2h_s\geq (\alpha^{-1}+(T-s))^{-1}$.
\end{lemma}
Then, if we assume that there exists some $\alpha>-T^{-1}$ such that $\nabla^2 h\geq \alpha$, the previous result implies that $\nabla^2 h_0\geq (\alpha^{-1}+T)^{-1}$ and hence that the semiconcavity parameter $\Lambda$ of the function $g^y_{h_0}(z)\coloneqq \frac{|z-y|^2}{2}-T\,h_0(z)$ can be bounded by
\begin{equation}\label{eq:general:Lambda0}\begin{aligned}
\Lam{h_0} \leq 1-T\,\lambda(h_0)
\leq 1-\frac{1}{(T\alpha)^{-1}+1}=\frac{1}{1+T\alpha}\,.
\end{aligned}\end{equation}

\subsection{Marginal $\rho$ with compact support}\label{sub:rho:compact}
Clearly if $\supp(\rho)\subseteq B_R(0)$  for some radius $R>0$, then for any $s\in[0,T)$ we have 
$\mathrm{Cov}(X^{\varphi^\nu,\nu}_T|X^{\varphi^\nu,\nu}_s=y)\leq R^2$ since $X^{\varphi^\nu,\nu}_T\sim \rho$ and as a consequence of \eqref{Hess:Cov:phi:s} we can take
\be\label{Lambda:phi:compact}
\lambda(\varphi^\nu_s)= 
(T-s)^{-1}-(T-s)^{-2}\,R^2\,,\text{ and hence }\Lam{\varphi^\nu_0}
=\nicefrac{R^2}{T}
\ee

Next, let us compute $\Cphi$ defined as
\bes
\Cphi\coloneqq T\biggl (\int_0^{T}e^{\int_0^s 2\lambda(\varphi^\nu_t)\De t}\De s\biggr)^{-1} \,.
\ees
This can be easily accomplished since for any $l\leq u< T$ we have
\bes
\begin{aligned}
\SymIpsi{l}{u}{\varphi^\nu}=&\, \int_l^u\exp\biggl(2\int_l^s\lambda(\varphi^\nu_t)\De t\biggr)\De s= \int_l^{u}\exp\biggl(2\int_l^s(T-t)^{-1}-(T-t)^{-2}\,R^2 \De t\biggr)\De s\\
=&\,\int_l^u\exp\biggl([-2\log( T-t)]_l^s-\biggl[\frac{2 R^2}{ T-t}\biggr]_l^s\biggr)\De s\\
    =&( T-l)^2\,e^{\frac{2 R^2}{ T-l}}\int_l^u\frac{e^{-\frac{2 R^2}{ T-s}}}{( T-s)^2}\De s=
        \frac{( T-l)^2}{2 R^2}\biggl(1-e^{\frac{2 R^2}{ T-l}-\frac{2 R^2}{ T-u}}\biggr)\,.
\end{aligned}
\ees
Therefore we have
\be\label{Cphi:compact}
\Cphi\coloneqq T\, \inf_{\delta\in[0.1)}(\SymIpsi{0}{\delta T}{\varphi^\nu})^{-1}=\frac{2R^2}{T}\,\inf_{\delta\in[0,1)}\biggl(1-\exp\biggl(-\frac{\delta}{1-\delta}\,\frac{2R^2}{T}\biggr)\biggr)^{-1}=\frac{2R^2}{T}\,.
\ee

\subsection{Log-concavity of $\rho$}\label{app:rho:logconcave:no:beta}
Let $U_\rho$ denotes the (negative) log-density of the marginal $\rho$ and let us assume that there exists $\alpha_\rho>0$ such that $
 \nabla^2U_\rho\geq \alpha_\rho$. Without loss of generalities, since we are interested in the asymptotics $T\downarrow 0$, we will further assume that $\alpha_\rho<T^{-1}$.

Then, it is well known \cite{conforti2023Sinkhorn} that $\nabla^2\varphi^\nu\geq \alpha_\rho-T^{-1}$ and hence we can take  $
\lambda(\varphi^\nu_T)=\alpha_\rho-T^{-1}$. This is enough to deduce from \Cref{lemma:backpropagation:convexity} that  
\bes
\nabla^2\varphi^\nu_s\geq \frac1{(\alpha_\rho-T^{-1})^{-1}+T-s}
\ees
and hence that we can set
\bes
\lambda(\varphi^\nu_0)=\frac1{(\alpha_\rho-T^{-1})^{-1}+T}=\frac{\alpha_\rho-T^{-1}}{\alpha_\rho\,T}<0\quad\text{ and hence }\Lam{\varphi^\nu_0}=(\alpha_\rho\,T)^{-1}\,,
\ees
and for any $s\in[0,T]$
\bes
\lambda(\varphi^\nu_s)=\frac1{\lambda(\varphi^\nu_0)^{-1}-s}<0 \,.
\ees
We are now ready to compute 
\bes
\Cphi\coloneqq T\biggl (\int_0^{T}e^{\int_0^s 2\lambda(\varphi^\nu_t)\De t}\De s\biggr)^{-1} \,.
\ees
This can be easily accomplished since for any $l\leq u< T$ we have
\bes
\begin{aligned}
\SymIpsi{l}{u}{\varphi^\nu}=&\, \int_l^u\exp\biggl(2\int_l^s\lambda(\varphi^\nu_t)\De t\biggr)\De s= \int_l^{u}\exp\biggl(-2\int_l^s\frac1{t-\lambda(\varphi^\nu_0)^{-1}}\De t\biggr)\\
=&\,\int_l^u\exp\biggl([-2\log(t-\lambda(\varphi^\nu_0)^{-1})]_l^s\biggr)\De s=\int_l^u\frac{(l-\lambda(\varphi^\nu_0)^{-1})^2}{(s-\lambda(\varphi^\nu_0)^{-1})^2}\De s\\
=&\,(l-\lambda(\varphi^\nu_0)^{-1})^2\biggl(\frac1{l-\lambda(\varphi^\nu_0)^{-1}}-\frac1{u-\lambda(\varphi^\nu_0)^{-1}}\biggr)\,.
\end{aligned}
\ees
Therefore we have
\be\label{Cphi:beta:oo}
\Cphi\coloneqq T\, \inf_{\delta\in[0.1)}(\SymIpsi{0}{\delta T}{\varphi^\nu})^{-1}=T\SymIpsi{0}{ T}{\varphi^\nu})^{-1}=(\alpha_\rho\,T)^{-1}\,.
\ee

\bigskip

Let us conclude this appendix with a table summarizing the values of the constants so far computed (up to numerical prefactors).

\medskip

\begin{center}
    
\begin{tabular}{|c|c|c|}
\hline
\textbf{Constant} & $\Lam{\varphi^\nu_0}$ & $\Cphi$ \\
\hline
     $\rho$ compact support &  $R^2\,T^{-1}$ & $R^2\,T^{-1}$\\
     \hline
     $\rho$ log-concave &  $\alpha_\rho^{-1}\,T^{-1}$ &  $\alpha_\rho^{-1}\,T^{-1}$ \\
\hline
\end{tabular}

\end{center}

 \bigskip

\section{Explicit computations for the stability of Hessians}\label{sec:app:psi}

In this section we will compute the constants appearing in Theorem 3.1 and Theorem 3.6 in two specific settings and analyze their behavior w.r.t. the parameters $T,\nu,\rho$ . Hereafter we write $a\lesssim b$ whenever there exists a numerical constant $C>0$ (independent of $T,\nu,\rho$) such that $a\leq C\,b$.
In order to compute the constants appearing in the stability bounds for the Hessian recall that 
 $\Cpsi_{\delta',\delta}$ was introduced as
 \be\label{rel:Cdelta:primo:delta:int}
\Cpsi_{\delta',\delta}=T(\SymIpsi{\delta' T}{\delta T}{\psi^\nu})^{-1}\,.
 \ee
Through this section we always choose 
\be\label{choice:delta}
\delta=\frac1{1+\Lam{\psi^\nu_0}}\,,
\ee
so that 
\be\label{eq:delta:1-delta:Lambda}
\frac\delta{1-\delta}=\frac{1}{\Lam{\psi^\nu_0}}\quad\text{ and }\quad\frac1{1-\delta}=\frac{1+\Lam{\psi^\nu_0}}{\Lam{\psi^\nu_0}}\,.
\ee
Moreover we will pick $\delta'=\delta/2$ so that
\bes
\frac{\delta'}{1-\delta'}=((\delta')^{-1}-1)^{-1}=(\nicefrac2\delta-1)^{-1}=(1+2\Lam{\psi^\nu_0})^{-1}
\ees
Finally, recall that hereafter we choose $\tau_u=\delta T$ and $\tau_\ell=\delta' T$ and note that in general we always have 
\be\label{eq:general:HS}
\begin{aligned}
\gamma_{\tau_u}\coloneqq \sup_{s\in[0,\tau_u]}\sup_{x\in\bbRD}\HS{\nabla^2\psi^\nu_s}\overset{\eqref{Hess:Cov:psi:s}}{\leq} \sup_{s\in[0,\tau_u]} \sqrt{d}\,(T-s)^{-1}=\frac{\sqrt{d}}{T(1-\delta)}=\frac{\sqrt{d}}{T}\,\frac{1+\Lam{\psi^\nu_0}}{\Lam{\psi^\nu_0}}\,.
\end{aligned}\ee

\subsection{Marginal $\nu$ with compact support}
By reasoning as in \Cref{sub:rho:compact}, if $\supp(\nu)\subseteq B_R(0)$ for some radius $R>0$, which we assume to be big enough, \ie, that $R^2\geq T$. 
Then for any $s\in[0,T)$ we have 
$\mathrm{Cov}(X^{\psi^\nu,\rho}_T|X^{\psi^\nu,\rho}_s=y)\leq R^2$ since $X^{\psi^\nu,\rho}_T\sim \nu$ and as a consequence of \eqref{Hess:Cov:psi:s} we can take
\be\label{Lambda:psi:s:compact}
\lambda(\psi^\nu_s)=
(T-s)^{-1}-(T-s)^{-2}\,R^2\,,
\ee
and hence
\be\label{eq:I:compatto}
       \SymIpsi{l}{u}{\psi^\nu}= \frac{( T-l)^2}{2 R^2}\biggl(1-e^{\frac{2 R^2}{ T-l}-\frac{2 R^2}{ T-u}}\biggr)\,,\quad\text{ and we take }\Lam{\psi^\nu_0}
       =\nicefrac{R^2}{T}\,.
\ee
This combined with \eqref{eq:general:HS} already gives
\bes
\gamma_{\tau_u}\leq\frac{\sqrt{d}}{T}\,\frac{1+\Lam{\psi^\nu_0}}{\Lam{\psi^\nu_0}}
=\sqrt{d} (R^{-2}+T^{-1})\leq \frac{2\sqrt{d}}{T}\,.
\ees
Next let us compute the integral constant term appearing in Theorem 3.6, that is the value
\bes
\sup_{s\in[0,\tau_\ell]}\int_s^{\tau_u}\SymIpsi{s}{u}{\psi^\nu}^{-\nicefrac12}\De u \,.
\ees
 In view of that, notice that for any $s\in[0,\tau_\ell]$
\bes
\begin{aligned}
\int_s^{\tau_u}&\SymIpsi{s}{u}{\psi^\nu}^{-\nicefrac12}\De u = \sqrt{2}\,\frac{R}{T-s}\int_s^{\tau_u} \biggl(1-e^{\frac{2 R^2}{ T-s}-\frac{2 R^2}{T-u}}\biggr)^{-\nicefrac12}\De u\\
\leq&\, \sqrt{2}\,\frac{R}{T-s}\int_s^{\tau_u} \biggl(1-e^{-\frac{2 R^2}{(T-s)^2}(u-s)}\biggr)^{-\nicefrac12}\De u
= \frac{T-s}{\sqrt{2}\,R}\,\log\left(\frac{1+\sqrt{1-e^{-\frac{2R^2}{(T-s)^2}(\tau_u-s)}}}{1-\sqrt{1-e^{-\frac{2R^2}{(T-s)^2}(\tau_u-s)}}}\right)\\
\leq&\,\frac{\log 4}{\sqrt{2}}\,\frac{T-s}{R}+ \sqrt{2}\,R\,\frac{\tau_u-s}{T-s} \leq \frac{\log 4}{\sqrt{2}}\,\frac{T}{R}+ \sqrt{2}\,R\,\frac{\tau_u}{T-\tau_\ell}= \frac{\log 4}{\sqrt{2}}\,\frac{T}{R}+ \sqrt{2}\,R\,\frac{2\delta'}{1-\delta'}\\
=& \,\frac{\log 4}{\sqrt{2}}\,\frac{T}{R}+ \frac{2\sqrt{2}\,R}{1+2\Lam{\psi^\nu_0}}=\frac{\log 4}{\sqrt{2}}\,\frac{T}{R}+ \frac{2\sqrt{2}\,T\,R}{T+2\,R^2}\,.
\end{aligned}\ees
Therefore
\be
\sup_{s\in[0,\tau_\ell]}\int_s^{\tau_u}\SymIpsi{s}{u}{\psi^\nu}^{-\nicefrac12}\De u\lesssim \frac{T}{R}+ \frac{T\,R}{T+R^2}\lesssim \nicefrac{T}{R}\,.
\ee

\medskip

Now, let us compute  $C^{\psi^\nu}_{\delta',\delta}$ from \eqref{rel:Cdelta:primo:delta:int} and \eqref{eq:I:compatto}. We have
\bes
\begin{aligned}
      C^{\psi^\nu}_{\delta',\delta}= \frac{2 \,R^2}{T(1-\delta')^2}\biggl(1-e^{\frac{2 R^2}{ T(1-\delta')}-\frac{2 R^2}{ T(1-\delta)}}\biggr)^{-1}=\frac{2 \,R^2}{T(1-\delta')^2}\biggl(1-\exp\biggl(-\frac{ R^2}{ T}\,\frac{\delta}{(1-\delta)(1-\delta')}\biggr)\biggr)^{-1}\\
      \overset{\eqref{eq:delta:1-delta:Lambda}}{=}\frac{2 \,R^2}{T(1-\delta')^2}\biggl(1-\exp\biggl(-\frac{ R^2}{ T}\,\frac1{\Lam{\psi^\nu_0}(1-\delta')}\biggr)\biggr)^{-1}=\frac{2 \,R^2}{T(1-\delta')^2}\biggl(1-\exp\biggl(-\frac1{1-\delta'}\biggr)\biggr)^{-1}\\
      \leq \frac{2 \,R^2}{T(1-\delta)^2}\frac1{1-e^{-1}}=\frac{(1+\Lam{\psi^\nu_0})^2}{\Lam{\psi^\nu_0}^2}\,\frac{R^2}{T}\,\frac{2 }{1-e^{-1}}=(1+\nicefrac{R^2}{T})^2\,\frac{2 }{1-e^{-1}}\lesssim 1+\nicefrac{R^4}{T^2}\leq \nicefrac{R^4}{T^2} \,.
\end{aligned}
\ees
Similarly,  we can compute
\bes
\Cpsi_{\delta}=\Cpsi_{0,\delta}= \frac{2}{1-e^{-1}}\,\frac{R^2}{T}
\ees
Lastly, notice that from $R^2\geq T$ we know that $\lambda(\psi^\nu_s)\leq 0$ and it is monotone decreasing, which yields to
\bes
\begin{aligned}
\bar\lambda_{\tau_\ell}\coloneqq (\inf_{s\in[0,{\tau_\ell}]}\lambda(\psi^\nu_s))^-=-\lambda(\psi^\nu_{\tau_\ell})=\frac{R^2}{T^2(1-\delta')^2}-\frac{1}{T(1-\delta')} \leq \frac{R^2}{T^2(1-\delta')^2}\\
=\frac{R^2}{T^2}\biggl(\frac{1+\Lam{\psi^\nu_0}}{\nicefrac12+\Lam{\psi^\nu_0}}\biggr)^2\leq 4\,\frac{R^2}{T^2}\,.
\end{aligned}
\ees

\medskip

\subsection{Log-concavity of $\nu$}\label{app:nu:logconcave:no:beta}
By reasoning as in \Cref{app:rho:logconcave:no:beta}, if  $U_\nu$ denotes the (negative) log-density of~$\nu$ and we assume that $\nabla^2U_\nu\geq \alpha_\nu$ for some  $\alpha_\nu>0$ (w.l.o.g.\ such that $\alpha_\nu<T^{-1}$) then we can consider
\bes
\lambda(\psi^\nu_s)=\frac1{\lambda(\psi^\nu_0)^{-1}-s} \quad\text{ where }\quad
\lambda(\psi^\nu_0)=\frac1{(\alpha_\psi-T^{-1})^{-1}+T}=\frac{\alpha_\nu-T^{-1}}{\alpha_\nu\,T}\,,
\ees
since for any $s\in[0,T]$ it holds
 \bes
\nabla^2\psi^\nu_s\geq \frac1{(\alpha_\nu-T^{-1})^{-1}+T-s}\,.
\ees
Moreover, this further implies  $\Lam{\psi^\nu_0}=(\alpha_\nu\,T)^{-1}$, and since $\alpha_\nu<T^{-1}$ we are guaranteed that $\lambda(\psi^\nu_s)$ is always negative. This combined with \eqref{eq:general:HS} already gives
\bes
\gamma_{\tau_u}\leq\frac{\sqrt{d}}{T}\,\frac{1+\Lam{\psi^\nu_0}}{\Lam{\psi^\nu_0}}
=\sqrt{d} (\alpha_\nu+T^{-1})\,.
\ees
Next, by reasoning as in \Cref{app:rho:logconcave:no:beta} we have
\be\label{eq:I:log:concave:no:beta}
\SymIpsi{l}{u}{\psi^\nu}=(l-\lambda(\psi^\nu_0)^{-1})^2\biggl(\frac1{l-\lambda(\psi^\nu_0)^{-1}}-\frac1{u-\lambda(\psi^\nu_0)^{-1}}\biggr)=\frac{l-\lambda(\psi^\nu_0)^{-1}}{u-\lambda(\psi^\nu_0)^{-1}}\,(u-l)\,,
\ee
and hence that for any $s\in[0,\tau_\ell]$
\bes
\begin{aligned}
\int_s^{\tau_u}\SymIpsi{s}{u}{\psi^\nu}^{-\nicefrac12}\De u = \int_s^{\tau_u}\sqrt{\frac{u-\lambda(\psi^\nu_0)^{-1}}{s-\lambda(\psi^\nu_0)^{-1}}}\,\frac1{\sqrt{u-s}}\,\De u\leq \sqrt{\frac{\tau_u-\lambda(\psi^\nu_0)^{-1}}{s-\lambda(\psi^\nu_0)^{-1}}}\,\int_s^{\tau_u}\frac1{\sqrt{u-s}}\,\De u\\
=2\sqrt{\frac{\tau_u-\lambda(\psi^\nu_0)^{-1}}{s-\lambda(\psi^\nu_0)^{-1}}}\,\sqrt{\tau_u-s}\leq 2\,\sqrt{\tau_u}\,\sqrt{1-\tau_u\,\lambda(\psi^\nu_0)}=2\sqrt{2}\,\frac{\sqrt{\alpha_\nu}\,T}{1+\alpha_\nu\,T}\,,
\end{aligned}\ees
and hence
\bes
\sup_{s\in[0,\tau_\ell]}\int_s^{\tau_u}\SymIpsi{s}{u}{\psi^\nu}^{-\nicefrac12}\De u\leq\frac{2\sqrt{2}}{\sqrt{\alpha_\nu}}\,.
\ees
Next, from \eqref{rel:Cdelta:primo:delta:int} and \eqref{eq:I:log:concave:no:beta} we may compute  $C^{\psi^\nu}_{\delta',\delta}$ and $\Cpsi_\delta=\Cpsi_{0,\delta}$ as
\bes
\begin{aligned}
      C^{\psi^\nu}_{\delta',\delta}= \frac{8}{\alpha_\nu\,T}\,\frac{1+\alpha_\nu\,T}{3+\alpha_\nu\,T}\leq 8\,(\alpha_\nu\,T)^{-1}\quad\text{ and }\quad \Cpsi_\delta=\frac{2}{\alpha_\nu \,T}\,.
\end{aligned}
\ees
Lastly, notice that  $\lambda(\psi^\nu_s)$ is a negative monotone increasing sequence and hence 
\bes
\begin{aligned}
\bar\lambda_{\tau_\ell}\coloneqq (\inf_{s\in[0,{\tau_\ell}]}\lambda(\psi^\nu_s))^-=-\lambda(\psi^\nu_0)=\frac{T^{-1}-\alpha_\nu}{\alpha_\nu\,T}\,.
\end{aligned}
\ees

\medskip

In the following table we summarize the values of the constants so far computed (up to numerical prefactors).

\begin{center}
\resizebox{.96\textwidth}{!}{ 
    
\begin{tabular}{|c|c|c|c|c|c|c|}
\hline
\textbf{Constant} & $\Lam{\psi^\nu_0}$ &   $C^{\psi^\nu}_{\delta',\delta}$ & $\Cpsi_\delta $ & $\gamma_{\tau_u}$  & $\sup_{s\in[0,\tau_\ell]}\int_s^{\tau_u}\SymIpsi{s}{u}{\psi^\nu}^{-\nicefrac12}\De u$ & $\bar\lambda_{\tau_\ell}$  \\
\hline
     $\nu$ compact support  & $R^2\,T^{-1}$ & $R^4\,T^{-2}$ & $R^2\,T^{-1}$ & $\sqrt{d}\,T^{-1}$ & $T\,R^{-1}$ & $R^2\,T^{-2}$ \\
     \hline
     $\nu$ log-concave &  $\alpha_\nu^{-1}\,T^{-1}$ &  $\alpha_\nu^{-1}\,T^{-1}$ & $\alpha_\nu^{-1}\,T^{-1}$ & $\sqrt{d}(\alpha_\nu+ T^{-1})$ & $\alpha_\nu^{-\nicefrac12}$ & $\alpha_\nu^{-1}T^{-2}-T^{-1}$ \\
\hline
\end{tabular}
}

\end{center}

\medskip

Let us conclude by specifying our stability results of gradients and Hessians from Theorem 3.1 and Theorem 3.6 to the two settings considered above, relying on the explicit computations performed in \Cref{sec:app:phi} and \Cref{sec:app:psi}. This will prove the asymptotic bounds stated in Theorem 1.1 in the main article. Recall that hereafter we write $a\lesssim b$ whenever there exists a numerical constant $C>0$ (independent of $T,\nu,\rho,\mu$) such that $a\leq C\,b$.

\begin{corollary}[Stability of gradients for compactly supported marginals]\label{cor:grad:compact}
Assume \Cref{ass:basic}, that $\scrH(\nu|\Leb)<\infty$, that both $\rho$ and $\nu$ are compactly supported in a ball of radius $R$ (big enough so that $R^2\geq T$) and that either $\mu\ll\nu$ or $\supp(\mu)\subseteq B_R(0)$. Then we have
\bes
\|\nabla\varphi^\nu-\nabla\varphi^\mu \|^2_{\rmL^2(\rho)} \lesssim \frac{R^4}{T^4}\,\bfW_2^2(\mu,\nu)\,.
\ees
\end{corollary}
\begin{proof}
   Firstly, assume $\mu\ll\nu$. Since $\rho$ has compact support, the same computations performed in \Cref{sub:rho:compact} guarantee $\Lam{\varphi^\mu_0}<\infty$ and hence the validity of the gradient estimates from Theorem 3.1. Our choice of $\delta$ in~\eqref{choice:delta} and the following computations yield to $\Crhonu^\delta\lesssim R^4\,T^{-2}$.

   Next, if we assume $\supp(\mu)\subseteq B_R(0)$ instead of $\mu\ll\nu$ we proceed as follows. Consider the probability measure $\nu^\eta\propto\mu+\eta\nu$ and observe that $\nu\ll\nu^\eta$ and $\mu\ll\nu^\eta$. The former implies the applicability of the previous case to the pair $(\nu^\eta,\nu)$, which guarantees $\nabla\varphi^{\nu^\eta}\rightarrow\nabla\varphi^\nu$  in $\rmL^2(\rho)$ as the regularization parameter $\eta$ vanishes. The latter ensures the validity of quantitative stability bound for the pair $(\mu,\nu^\eta)$, uniformly in $\eta>0$, since $\supp(\nu^\eta)\subseteq B_R(0)$ for all $\eta\geq 0$. Thus, it is sufficient to apply the triangle inequality and eventually send $\eta\downarrow 0$ to recover the final stability bound for the original pair $(\mu,\nu)$. 
\end{proof}

\begin{corollary}[Stability of gradients for log-concave marginals]\label{cor:grad:log:conc}
Assume \Cref{ass:basic}, $\scrH(\nu|\Leb)<\infty$  and that both $\rho$ and $\nu$ are log-concave, \ie, that their (negative) log-densities satisfy $\nabla^2U_\rho\geq \alpha_\rho$ and $\nabla^2U_\nu\geq \alpha_\nu$ for some  $\alpha_\rho,\,\alpha_\nu>0$ (w.l.o.g.\ such that $\alpha_\rho\vee\alpha_\nu<T^{-1}$). Then we have
\bes
\|\nabla\varphi^\nu-\nabla\varphi^\mu \|^2_{\rmL^2(\rho)} \lesssim \frac{1}{\alpha_\rho\,\alpha_\nu\,T^4}\,\bfW_2^2(\mu,\nu)\,.
\ees
\end{corollary}

\begin{proof}
   Firstly, let us consider the case where $\mu\ll\nu$. Since $\rho$ has log-concave density, the same computations performed in \Cref{app:rho:logconcave:no:beta} guarantee $\Lam{\varphi^\mu_0}<\infty$ and hence the validity of  the validity of the gradient estimates from Theorem 3.1. Our choice of $\delta$ in~\eqref{choice:delta} and the following computations yield to $\Crhonu^\delta\lesssim \alpha_\rho^{-1}\,\alpha_\nu^{-1}\,T^{-2}$.

  To lift the assumption $\mu\ll\nu$ we argue in the following way. Fix a regularization parameter $\eta>0$ and consider the convolution with the heat kernel $\nu^\eta$. Then, as the (negative) log-density of $\nu^\eta$ satisfies $\nabla^2 U_{\nu^\eta} \geq \frac{\alpha_\nu}{1+\eta\alpha_\nu}$ provided $\nabla^2 U_{\nu} \geq \alpha_\nu$ (see for instance \cite[Theorem 1]{henningsson2006log}), the log-concavity parameter of $\nu^\eta$ converges to the one of $\nu$ as $\eta \downarrow 0$. Since $\nu\ll\nu^\eta$, we can thus apply the previous result and deduce that as $\eta\downarrow 0$ we have
$\nabla\varphi^{\nu^\eta}\rightarrow\nabla\varphi^\nu$ in $\rmL^2(\rho)$. 
     Moreover, since $\nu^\eta\sim\Leb$ and $\mu\ll\Leb$, we can again apply the previous result to the pair $(\mu,\nu^\eta)$. Then, it is enough sending $\eta\downarrow 0$ to recover the final stability bound for the original pair $(\mu,\nu)$.
\end{proof}

Similarly, for the Hessian bounds we have the followings.

\begin{corollary}[Stability of Hessians for compactly supported marginals]\label{cor:hess:compact}
    Assume \Cref{ass:basic}, that $\scrH(\nu|\Leb)<\infty$, that both $\rho$ and $\nu$ are compactly supported in a ball of radius $R$ (big enough so that $R^2\geq T$) and that either $\mu\ll\nu$ or $\supp(\mu)\subseteq B_R(0)$. Then we have
    \bes 
    \| \nabla^2 \varphi^\mu -  \nabla^2 \varphi^\nu\|_{\rmL^1(\rho)} \lesssim (\nicefrac{R^4}{T^{\nicefrac72}}+\nicefrac{d}{T})\,\bfW_2(\mu,\nu)+
   \nicefrac{R^6}{T^5}\,\bfW_2^2(\mu,\nu)\,, 
    \ees
\end{corollary}
\begin{proof} Firstly, we prove this result in the case $\mu\ll\nu$. Since $\rho$ has compact support, the same computations performed in \Cref{sub:rho:compact} guarantee $\Lam{\varphi^\mu_0}<\infty$ and hence the validity of  of our general stability estimates for the Hessian of Schr\"odinger potentials  from Theorem 3.6. Our computations yield to 
    \bes
\Kappadel\lesssim \nicefrac{R^6}{T^3}\quad \text{ and }\quad A\lesssim \nicefrac{R^4}{T^{\nicefrac72}}+\nicefrac{d}{T}\,.
    \ees
    The proof of this result when considering the assumption $\supp(\mu)\subseteq B_R(0)$ (instead of $\mu\ll\nu$) can be obtained via the same regularization procedure described in the proof of~\Cref{cor:grad:compact} and for this reason we omit it here.
\end{proof}

\begin{corollary}[Stability of Hessians for log-concave marginals]\label{cor:hess:log:conc}
    Assume \Cref{ass:basic}, $\scrH(\nu|\Leb)<\infty$  and that both $\rho$ and $\nu$ are log-concave, \ie, that their (negative) log-densities satisfy $\nabla^2U_\rho\geq \alpha_\rho$ and $\nabla^2U_\nu\geq \alpha_\nu$ for some  $\alpha_\rho,\,\alpha_\nu>0$ (w.l.o.g.\ such that $\alpha_\rho\vee\alpha_\nu<T^{-1}$). Then we have
    \bes 
    \| \nabla^2 \varphi^\mu -  \nabla^2 \varphi^\nu\|_{\rmL^1(\rho)} \lesssim \biggl(\frac1{\alpha_\nu\,\sqrt{\alpha_\rho}\,T^3}+\frac{d}{\sqrt{\alpha_\rho\,\alpha_\nu}\,T^2}\biggr)\,\bfW_2(\mu,\nu)+
   \frac1{\alpha_\rho\,\alpha_\nu\,T^4}\,\bfW_2^2(\mu,\nu)\,, 
    \ees
\end{corollary}
\begin{proof}
We prove this result under the additional assumption $\mu\ll\nu$. The general result can be obtained following the same regularization procedure considered in the proof of \Cref{cor:grad:log:conc}. Since $\rho$ is log-concave, the same computations performed in \Cref{app:rho:logconcave:no:beta} guarantee $\Lam{\varphi^\mu_0}<\infty$ and hence the validity of our general stability estimates for the Hessian of Schr\"odinger potentials from Theorem 3.6. Our computations yield to  
    \bes
\Kappadel\lesssim \frac1{\alpha_\rho\,\alpha_\nu\,T^2}\quad \text{ and }\quad A\lesssim \frac1{\alpha_\nu\,\sqrt{\alpha_\rho}\,T^3}+\frac{d}{\sqrt{\alpha_\rho\,\alpha_\nu}\,T^2}\,.
    \ees
\end{proof}

\bibliographystyle{alpha}

\bibliography{arxiv-biblio-final.bbl}

\newcommand{\etalchar}[1]{$^{#1}$}
\begin{thebibliography}{SDBDD22}

\bibitem[ADMM24]{akyildiz2024gaussianentropicoptimaltransport}
O.~Deniz Akyildiz, Pierre Del~Moral, and Joaquín Miguez.
\newblock Gaussian entropic optimal transport: {S}chr\"odinger bridges and the
  {S}inkhorn algorithm.
\newblock {\em arXiv preprint arXiv:2412.18432}, 2024.

\bibitem[ADMM25]{akyildiz2025contractionpropertiessinkhornsemigroups}
O.~Deniz Akyildiz, Pierre Del~Moral, and Joaquín Miguez.
\newblock On the contraction properties of {S}inkhorn semigroups.
\newblock {\em arXiv preprint arXiv:2503.09887}, 2025.

\bibitem[AFKL22]{aubin2022mirror}
Pierre-Cyril Aubin-Frankowski, Anna Korba, and Flavien L{\'e}ger.
\newblock Mirror {D}escent with {R}elative {S}moothness in {M}easure {S}paces,
  with application to {S}inkhorn and {EM}.
\newblock In Alice~H. Oh, Alekh Agarwal, Danielle Belgrave, and Kyunghyun Cho,
  editors, {\em Advances in Neural Information Processing Systems}, 2022.

\bibitem[BBD24]{bauerschmidt2024stochastic}
Roland Bauerschmidt, Thierry Bodineau, and Benoit Dagallier.
\newblock Stochastic dynamics and the {P}olchinski equation: an introduction.
\newblock {\em Probability Surveys}, 21:200--290, 2024.

\bibitem[BCC{\etalchar{+}}15]{benamou2015iterative}
Jean-David Benamou, Guillaume Carlier, Marco Cuturi, Luca Nenna, and Gabriel
  Peyr{\'e}.
\newblock Iterative {B}regman projections for regularized transportation
  problems.
\newblock {\em SIAM Journal on Scientific Computing}, 37(2):A1111--A1138, 2015.

\bibitem[Ber20]{berman2020sinkhorn}
Robert~J Berman.
\newblock The {S}inkhorn algorithm, parabolic optimal transport and geometric
  {M}onge--{A}mp{\`e}re equations.
\newblock {\em Numerische Mathematik}, 145(4):771--836, 2020.

\bibitem[BGL13]{bakry2013analysis}
Dominique Bakry, Ivan Gentil, and Michel Ledoux.
\newblock {\em Analysis and geometry of Markov diffusion operators}, volume
  348.
\newblock Springer Science \& Business Media, 2013.

\bibitem[BGN22]{BerntonGhosalNutz}
Espen Bernton, Promit Ghosal, and Marcel Nutz.
\newblock {E}ntropic {O}ptimal {T}ransport: {G}eometry and {L}arge
  {D}eviations.
\newblock {\em Duke Mathematical Journal}, 171(16):3363 -- 3400, 2022.

\bibitem[BLN94]{borwein1994entropy}
Jonathan~M Borwein, Adrian~Stephen Lewis, and Roger Nussbaum.
\newblock Entropy minimization, {DAD} problems, and doubly stochastic kernels.
\newblock {\em Journal of Functional Analysis}, 123(2):264--307, 1994.

\bibitem[BTHD21]{bortoli2021diffusion}
Valentin~De Bortoli, James Thornton, Jeremy Heng, and Arnaud Doucet.
\newblock Diffusion schr\"odinger bridge with applications to score-based
  generative modeling.
\newblock In A.~Beygelzimer, Y.~Dauphin, P.~Liang, and J.~Wortman Vaughan,
  editors, {\em Advances in Neural Information Processing Systems}, 2021.

\bibitem[Car22]{Carlier22multisink}
Guillaume Carlier.
\newblock On the {L}inear {C}onvergence of the {M}ultimarginal {S}inkhorn
  {A}lgorithm.
\newblock {\em SIAM Journal on Optimization}, 32(2):786--794, 2022.

\bibitem[CC24]{chaintron2025regularitystabilitygibbsconditioning}
Louis-Pierre Chaintron and Giovanni Conforti.
\newblock Regularity and stability for the {G}ibbs conditioning principle on
  path space via {M}c{K}ean-{V}lasov control.
\newblock {\em arXiv preprint arXiv:2410.23016}, 2024.

\bibitem[CCGT23]{lagg2022gradient}
Alberto Chiarini, Giovanni Conforti, Giacomo Greco, and Luca Tamanini.
\newblock Gradient estimates for the {S}chr{\"o}dinger potentials: convergence
  to the {B}renier map and quantitative stability.
\newblock {\em Communications in Partial Differential Equations},
  48(6):895--943, 2023.

\bibitem[CCGT24]{chiarini2024semiconcavity}
Alberto Chiarini, Giovanni Conforti, Giacomo Greco, and Luca Tamanini.
\newblock A semiconcavity approach to stability of entropic plans and
  exponential convergence of {S}inkhorn's algorithm.
\newblock {\em arXiv preprint arXiv:2412.09235}, 2024.

\bibitem[CCL24]{carlier2024displacement}
Guillaume Carlier, L{\'e}na{\"\i}c Chizat, and Maxime Laborde.
\newblock Displacement smoothness of entropic optimal transport.
\newblock {\em ESAIM: Control, Optimisation and Calculus of Variations}, 30:25,
  2024.

\bibitem[CDG23]{conforti2023Sinkhorn}
Giovanni Conforti, Alain~Oliviero Durmus, and Giacomo Greco.
\newblock Quantitative contraction rates for {S}inkhorn algorithm: beyond
  bounded costs and compact marginals.
\newblock {\em arXiv preprint arXiv:2304.04451}, 2023.

\bibitem[CDV25]{chizat2025sharper}
L{\'e}na{\"\i}c Chizat, Alex Delalande, and Tomas Va{\v{s}}kevi{\v{c}}ius.
\newblock Sharper exponential convergence rates for {S}inkhorn’s algorithm in
  continuous settings.
\newblock {\em Mathematical Programming}, pages 1--50, 2025.

\bibitem[CE22]{chen2022localization}
Yuansi Chen and Ronen Eldan.
\newblock Localization schemes: {A} framework for proving mixing bounds for
  {M}arkov chains.
\newblock In {\em 2022 IEEE 63rd Annual Symposium on Foundations of Computer
  Science (FOCS)}, pages 110--122. IEEE, 2022.

\bibitem[CGP16]{chen2016hilbertmetric}
Yongxin Chen, Tryphon Georgiou, and Michele Pavon.
\newblock Entropic and {D}isplacement {I}nterpolation: {A} {C}omputational
  {A}pproach {U}sing the {H}ilbert {M}etric.
\newblock {\em SIAM Journal on Applied Mathematics}, 76(6):2375--2396, 2016.

\bibitem[CL20]{carlier2020differential}
Guillaume Carlier and Maxime Laborde.
\newblock A {D}ifferential {A}pproach to the {M}ulti-{M}arginal
  {S}chr{\"o}dinger {S}ystem.
\newblock {\em SIAM Journal on Mathematical Analysis}, 52(1):709--717, 2020.

\bibitem[Con24]{conforti2024weak}
Giovanni Conforti.
\newblock Weak semiconvexity estimates for {S}chr{\" o}dinger potentials and
  logarithmic {S}obolev inequality for {S}chr{\" o}dinger bridges.
\newblock {\em Probability Theory and Related Fields}, 189:1045--1071, 2024.

\bibitem[CP23]{chewi2022entropic}
Sinho Chewi and Aram-Alexandre Pooladian.
\newblock An entropic generalization of {C}affarelli's contraction theorem via
  covariance inequalities.
\newblock {\em Comptes Rendus. Mathématique}, 361:1471--1482, 2023.

\bibitem[Cut13]{cuturi2013sinkhorn}
Marco Cuturi.
\newblock Sinkhorn distances: Lightspeed computation of optimal transport.
\newblock In {\em Advances in {N}eural {I}nformation {P}rocessing {S}ystems},
  pages 2292--2300, 2013.

\bibitem[DdBD24]{deligiannidis2021quantitative}
George Deligiannidis, Valentin de~Bortoli, and Arnaud Doucet.
\newblock {Quantitative uniform stability of the iterative proportional fitting
  procedure}.
\newblock {\em The Annals of Applied Probability}, 34(1A):501 -- 516, 2024.

\bibitem[DKPS23]{deb2023wasserstein}
Nabarun Deb, Young-Heon Kim, Soumik Pal, and Geoffrey Schiebinger.
\newblock Wasserstein mirror gradient flow as the limit of the {S}inkhorn
  algorithm.
\newblock {\em Preprint, arXiv:2307.16421}, 2023.

\bibitem[DM25]{delmoral2025stabilityschrodingerbridgessinkhorn}
Pierre Del~Moral.
\newblock Stability of {S}chr\"odinger bridges and {S}inkhorn semigroups for
  log-concave models.
\newblock {\em arXiv preprint arXiv:2503.15963}, 2025.

\bibitem[DMG20]{marinogerolin2020}
Simone Di~Marino and Augusto Gerolin.
\newblock An {O}ptimal {T}ransport {A}pproach for the {S}chrödinger {B}ridge
  {P}roblem and {C}onvergence of {S}inkhorn {A}lgorithm.
\newblock {\em Journal of Scientific Computing}, 85(2):27, 2020.

\bibitem[DNWP25]{divol2024tight}
Vincent Divol, Jonathan Niles-Weed, and Aram-Alexandre Pooladian.
\newblock Tight stability bounds for entropic {B}renier maps.
\newblock {\em International Mathematics Research Notices}, 2025(7):rnaf078, 04
  2025.

\bibitem[Eck25]{eckstein2023hilberts}
Stephan Eckstein.
\newblock Hilbert's projective metric for functions of bounded growth and
  exponential convergence of {S}inkhorn's algorithm.
\newblock {\em Probability Theory and Related Fields}, 2025.

\bibitem[EL25]{eckstein2025exponentialconvergencegeneraliterative}
Stephan Eckstein and Aziz Lakhal.
\newblock Exponential convergence of general iterative proportional fitting
  procedures.
\newblock {\em arXiv preprint arXiv:2502.20264}, 2025.

\bibitem[EN22]{eckstein2021quantitative}
Stephan Eckstein and Marcel Nutz.
\newblock Quantitative {S}tability of {R}egularized {O}ptimal {T}ransport and
  {C}onvergence of {S}inkhorn’s {A}lgorithm.
\newblock {\em SIAM Journal on Mathematical Analysis}, 54(6):5922--5948, 2022.

\bibitem[FGP20]{fathi2019proof}
Max Fathi, Nathael Gozlan, and Maxime Prodhomme.
\newblock A proof of the {C}affarelli contraction theorem via entropic
  regularization.
\newblock {\em Calculus of Variations and Partial Differential Equations},
  59(96), 2020.

\bibitem[FL89]{Franklin89hilbert}
Joel Franklin and Jens Lorenz.
\newblock On the scaling of multidimensional matrices.
\newblock {\em Linear Algebra and its Applications}, 114-115:717--735, 1989.

\bibitem[GN25]{ghosal2022nutz}
Promit Ghosal and Marcel Nutz.
\newblock On the {C}onvergence {R}ate of {S}inkhorn’s {A}lgorithm.
\newblock {\em Mathematics of Operations Research}, forthcoming, 2025+.

\bibitem[GNB22]{ghosal2021stability}
Promit Ghosal, Marcel Nutz, and Espen Bernton.
\newblock Stability of entropic optimal transport and {S}chr{\"o}dinger
  bridges.
\newblock {\em Journal of Functional Analysis}, 283(9):109622, 2022.

\bibitem[GNCD23]{greco2023SinkhornTorus}
Giacomo Greco, Maxence Noble, Giovanni Conforti, and Alain Durmus.
\newblock Non-asymptotic convergence bounds for {S}inkhorn iterates and their
  gradients: a coupling approach.
\newblock In Gergely Neu and Lorenzo Rosasco, editors, {\em Proceedings of
  Thirty Sixth Conference on Learning Theory}, volume 195 of {\em Proceedings
  of Machine Learning Research}, pages 716--746. PMLR, 12--15 Jul 2023.

\bibitem[Gre24]{greco2024thesis}
Giacomo Greco.
\newblock {\em The Schr{\"o}dinger problem: where analysis meets stochastics}.
\newblock Phd {T}hesis ({G}raduation {TU}/e), Mathematics and Computer Science,
  may 2024.
\newblock Proefschrift.

\bibitem[H{\AA}06]{henningsson2006log}
Toivo Henningsson and Karl~Johan {\AA}str{\"o}m.
\newblock Log-concave observers.
\newblock In {\em 17th International Symposium on Mathematical Theory of
  Networks and Systems, 2006: MTNS 2006}, 2006.

\bibitem[KLM25]{kitagawa2025stabilityoptimaltransportmaps}
Jun Kitagawa, Cyril Letrouit, and Quentin Mérigot.
\newblock Stability of optimal transport maps on {R}iemannian manifolds.
\newblock {\em arXiv preprint arXiv:2504.05412}, 2025.

\bibitem[LAF23]{leger2023gradient}
Flavien L{\'e}ger and Pierre-Cyril Aubin-Frankowski.
\newblock Gradient descent with a general cost.
\newblock {\em Preprint, arXiv:2305.04917}, 2023.

\bibitem[L{\'e}g21]{leger2021gradient}
Flavien L{\'e}ger.
\newblock A gradient descent perspective on { S}inkhorn.
\newblock {\em Applied Mathematics {\&} Optimization}, 84(2):1843--1855, 2021.

\bibitem[L{\'e}o14]{LeoSch}
Christian L{\'e}onard.
\newblock A survey of the {S}chrödinger problem and some of its connections
  with optimal transport.
\newblock {\em Discrete and Continuous Dynamical Systems}, 34(4):1533--1574,
  2014.

\bibitem[LM24]{letrouit2024gluingmethodsquantitativestability}
Cyril Letrouit and Quentin Mérigot.
\newblock Gluing methods for quantitative stability of optimal transport maps.
\newblock {\em arXiv preprint arXiv:2411.04908}, 2024.

\bibitem[Mik04]{Mikami04}
Toshio Mikami.
\newblock Monge’s problem with a quadratic cost by the zero-noise limit of
  $h$-path processes.
\newblock {\em Probability Theory and Related Fields}, 129:245--260, 2004.

\bibitem[MS25]{malamut2023convergence}
Hugo Malamut and Maxime Sylvestre.
\newblock Convergence rates of the regularized optimal transport: Disentangling
  suboptimality and entropy.
\newblock {\em SIAM Journal on Mathematical Analysis}, 57(3):2533--2558, 2025.

\bibitem[MTW05]{MTW05}
Xi-Nan Ma, Neil~S Trudinger, and Xu-Jia Wang.
\newblock Regularity of potential functions of the optimal transportation
  problem.
\newblock {\em Archive for rational mechanics and analysis}, 177:151--183,
  2005.

\bibitem[Nut21]{Marcel:notes}
Marcel Nutz.
\newblock Introduction to {E}ntropic {O}ptimal {T}ransport.
\newblock {\em
  \url{http://www.math.columbia.edu/~mnutz/docs/EOT_lecture_notes.pdf}}, 2021.

\bibitem[NW22]{nutz2021entropic}
Marcel Nutz and Johannes Wiesel.
\newblock Entropic optimal transport: convergence of potentials.
\newblock {\em Probability Theory and Related Fields}, 184(1):401--424, 2022.

\bibitem[PNW21]{pooladian2021entropic}
Aram-Alexandre Pooladian and Jonathan Niles-Weed.
\newblock Entropic estimation of optimal transport maps.
\newblock {\em arXiv preprint arXiv:2109.12004}, 2021.

\bibitem[Rus95]{ruschendorf1995convergence}
Ludger Ruschendorf.
\newblock Convergence of the iterative proportional fitting procedure.
\newblock {\em The Annals of Statistics}, pages 1160--1174, 1995.

\bibitem[SABP22]{sander2022sinkformers}
Michael~E Sander, Pierre Ablin, Mathieu Blondel, and Gabriel Peyr{\'e}.
\newblock Sinkformers: {T}ransformers with doubly stochastic attention.
\newblock In {\em International Conference on Artificial Intelligence and
  Statistics}, pages 3515--3530. PMLR, 2022.

\bibitem[San15]{Santambrogio2015}
Filippo Santambrogio.
\newblock {\em Optimal Transport for Applied Mathematicians: Calculus of
  Variations, PDEs, and Modeling}, volume~87 of {\em Progress in Nonlinear
  Differential Equations and Their Applications}.
\newblock Birkhäuser, Cham, 2015.

\bibitem[Sch31]{Schr}
Erwin Schr{\"o}dinger.
\newblock {\"U}ber die {U}mkehrung der {N}aturgesetze.
\newblock {\em Sitzungsberichte Preuss. Akad. Wiss. Berlin. Phys. Math.},
  144:144--153, 1931.

\bibitem[Sch32]{Schr32}
Erwin Schr\"odinger.
\newblock La th\'eorie relativiste de l'\'electron et l' interpr\'etation de la
  m\'ecanique quantique.
\newblock {\em Ann. Inst Henri Poincar\'e}, (2):269 -- 310, 1932.

\bibitem[SDBDD22]{shi2022ConditionalSimviaSB}
Yuyang Shi, Valentin De~Bortoli, George Deligiannidis, and Arnaud Doucet.
\newblock Conditional simulation using diffusion {S}chr{ö}dinger bridges.
\newblock In James Cussens and Kun Zhang, editors, {\em Proceedings of the
  Thirty-Eighth Conference on Uncertainty in Artificial Intelligence}, volume
  180 of {\em Proceedings of Machine Learning Research}, pages 1792--1802.
  PMLR, 01--05 Aug 2022.

\bibitem[Sin64]{Sinkhorn64}
Richard Sinkhorn.
\newblock A {R}elationship {B}etween {A}rbitrary {P}ositive {M}atrices and
  {D}oubly {S}tochastic {M}atrices.
\newblock {\em The Annals of Mathematical Statistics}, 35(2):876--879, 1964.

\bibitem[SK67]{SinkhornKnopp67}
Richard Sinkhorn and Paul Knopp.
\newblock Concerning nonnegative matrices and doubly stochastic matrices.
\newblock {\em Pacific Journal of Mathematics}, 21(2):343--348, 1967.

\bibitem[WJX{\etalchar{+}}21]{wang2021DeepGenviaSB}
Gefei Wang, Yuling Jiao, Qian Xu, Yang Wang, and Can Yang.
\newblock Deep generative learning via schr{ö}dinger bridge.
\newblock In Marina Meila and Tong Zhang, editors, {\em Proceedings of the 38th
  International Conference on Machine Learning}, volume 139 of {\em Proceedings
  of Machine Learning Research}, pages 10794--10804. PMLR, 18--24 Jul 2021.

\end{thebibliography}

\end{document}